\documentclass[11pt]{amsart}
\usepackage{amsmath}
\usepackage{amsfonts}
\usepackage{amssymb}
\usepackage{amsthm}
\usepackage{indentfirst}
\usepackage{hyperref}
\usepackage{graphicx}
\usepackage{xcolor}

\usepackage[left=2.4cm,right=2.4cm,top=2.3cm,bottom=2.3cm]{geometry}
\title[Gel'fand's inverse problem under Ricci curvature bounds]{Gel'fand's inverse problem under Ricci curvature bounds}

\author{Shouhei Honda}
\address{Shouhei Honda: Graduate School of Mathematical Sciences, The University of Tokyo.}
\email{shouhei@ms.u-tokyo.ac.jp}

\author{Jinpeng Lu}
\address{Jinpeng Lu: School of Mathematical Sciences, University of Science and Technology of China, Hefei
230026, China, and, Department of Mathematics and Statistics, University of Helsinki.}
\email{jinpeng.lu@ustc.edu.cn}

\theoremstyle{definition}
\newtheorem{definition}{Definition}[section]

\newtheorem{remark}[definition]{Remark}

\theoremstyle{plain}
\newtheorem{theorem}[definition]{Theorem}
\newtheorem{lemma}[definition]{Lemma} 
 
\newtheorem{proposition}[definition]{Proposition} 
\newtheorem{corollary}[definition]{Corollary}

\numberwithin{equation}{section}

\newcommand{\setR}{\mathbb{R}}

\newcommand{\di}{\mathop{}\!\mathrm{d}}

\DeclareMathOperator{\supp}{supp}

\newcommand{\Ch}{{\sf Ch}}

\newcommand{\dist}{\mathsf{d}}

\newcommand{\meas}{\mathfrak{m}}

\DeclareMathOperator{\RCD}{RCD}

\newfont{\tmpf}{cmsy10 scaled 2500}

\newcommand{\intav}{{\mathop{\int\kern-10pt\rotatebox{0}{\textbf{--}}}}}

\renewcommand{\ }{\text{ }}
\def\<{\langle}
\def\>{\rangle}

\newcommand{\R}{\setR}

\newcommand{\diam}{\mathop{\mathrm{diam}}}

\begin{document}
\maketitle

\vspace{-2mm}

\begin{abstract}

The classical Gel'fand's inverse problem asks whether a Riemannian manifold is uniquely determined by the knowledge of the heat kernel on any open subset of the manifold. We study this inverse problem in the non-smooth setting in the framework of $\RCD(K,N)$ spaces, namely, metric-measure spaces with synthetic Riemannian Ricci curvature bounded below by $K$ and dimension bounded above by $N$.
We establish the unique solvability of Gel'fand's inverse problem for the class of compact $\RCD(K,N)$ spaces whose regular set admits Lipschitz Riemannian structure.
As an application, we obtain the continuity of Gel'fand's inverse problem in the class of closed Riemannian manifolds with bounded Ricci curvature, diameter and volume bounded from below. We note that the results are new even for Einstein orbifolds and (weighted) Lipschitz Riemannian manifolds with non-smooth boundary.

\end{abstract}

\section{introduction}

In the 1990s, Cheeger and Colding built an influential program \cite{CC1,CC2,CC3} on the structure of Ricci limit spaces. In particular, they proved the spectral convergence of the Laplacian under the measured Gromov-Hausdorff convergence in the class of $n$-dimensional closed Riemannian manifolds with Ricci curvature bounded from below $\textrm{Ric}\geq K$, extending Fukaya's result \cite{F87} for bounded sectional curvature to Ricci curvature. Later, by using this theory, Ding \cite{Ding} established the local uniform convergence of the heat kernels in the same setting.

Then a natural question is to examine the {\it inverse problem}: to what extent does the spectrum or the heat kernel determine the geometry?
While the spectrum alone cannot determine the geometry in general, as is well-known in Kac's question ``Can you hear the shape of a drum?", the heat kernel does provide sufficient information to determine Riemannian manifolds due to \cite{BBG,KK1,KK2}.
More precisely, let $M_1,M_2$ be two $n$-dimensional closed smooth Riemannian manifolds with ${\rm Ric}(M)\geq K$ and $\mathrm{diam}(M) \le D$.
For small $\delta>0$, suppose that there is a map $\psi:B\to M_2$ defined on a set $B\subset M_1$ satisfying 
\begin{equation}\label{eq:heat comparison-intro}
 1-\delta \le \frac{p_{M_2}(\psi(x), \psi(y), t)}{p_{M_1}(x, y, t)} \le 1+\delta, \quad \textrm{for all }x, y\in B,\, t\in [\delta,1],
\end{equation}
where $p_{M_i}$ denotes the heat kernels of $M_i$ for $i=1,2$.
If $B=M_1$, then $\dist_{\mathrm{GH}}(M_1, M_2)<\varepsilon_{n,K,D}(\delta)$ depending on $n,K,D$, see for instance Corollary \ref{cor:global rigidity} for a proof with notation (\ref{eq:small}). 


In fact, the stronger statement is true even when one only assumes the closeness of heat kernels via $\psi$ defined on any open subset $B\subset M_1$, provided that the manifolds are in the more restrictive class of bounded Ricci curvature, diameter and injectivity radius bounded from below.
This is known as the stability formulation of Gel'fand's inverse problem \cite{Ge} that was originally formulated in 1954 and solved by \cite{BK,T95,AKKLT}. 
The complete solution \cite{AKKLT} in the class above relies on the fact that the Gromov-Hausdorff limit spaces of this class are Riemannian manifolds (of lower regularity in the Zygmund class $C^2_*$).
The inverse problem is open in the general case.

\smallskip
{\bf Problem A.} {\it Let $\psi:B\to M_2$, defined only on an open subset $B\subset M_1$, be a map that almost preserves the heat kernels of two $n$-dimensional closed smooth Riemannian manifolds $M_1,M_2$ with ${\rm Ric}(M_i)\geq K$ and $\mathrm{diam}(M_i)\le D$ in the sense of \eqref{eq:heat comparison-intro}. Are $M_1,M_2$ close to each other in the Gromov-Hausdorff topology?}

\smallskip

It is worth mentioning that if we drop the assumption of Ricci curvature bound, then we can easily construct a counterexample as follows. Take the glued space $X$ of two unit spheres $\mathbb{S}^2_i$ for $i=1,2$ at each base point $p_i$ and a sequence of Riemannian metrics $g_i$ on $\mathbb{S}^2$, which looks like dumbbells with thin necks, Gromov-Hausdroff converging to $X$. Then considering a connected open subset of $X$ which is away from the glued point gives a desired counterexample (see also \cite[Ex. 3.8]{H18}).

To consider the general case, the first step is to relax the lower bound for the injectivity radius to the non-collapsed condition (namely, a lower bound for volume).
However, relaxing the injectivity radius condition already allows the limit spaces to be singular and the classical methods of inverse problems break down.
It is a major question whether the inverse problem can be studied in the singular setting. 
In this paper, we solve Problem A in the non-collapsed class of closed Riemannian manifolds with bounded Ricci curvature and diameter.

\smallskip
Let $\mathcal{M}(n, D, v)$ be the class of $n$-dimensional closed smooth Riemannian manifolds $M$ satisfying
\begin{equation}
|\textrm{Ric}(M)|\leq 1, \quad \textrm{diam}(M)\leq D, \quad \textrm{Vol}(M)\geq v>0,
\end{equation}
and let $\dist_{\mathrm{GH}}$ denote the Gromov-Hausdorff distance.

\smallskip
We are ready to state the first main result of the paper.

\begin{theorem}\label{thm:function-intro}
For all $n \in \mathbb{N}$ and $\epsilon, r, v, D \in (0, \infty)$, there exists $\delta=\delta(n, \epsilon, r, v, D)>0$ such that the following holds. For two closed Riemannian manifolds $M_i \in \mathcal{M}(n, D, v)$ for $i=1,2$, if there exists a map 
$\psi:B \to M_2$ on an open ball $B\subset M_1$ of radius $r$ satisfying (\ref{eq:heat comparison-intro}), 
then
\begin{equation}
\dist_{\mathrm{GH}}(M_1, M_2)<\epsilon.
\end{equation}
Furthermore, $\psi$ is obtained as the restriction of an $\epsilon$-Gromov-Hausdorff approximation $\Psi: M_1 \to M_2$ to $B$, with the following almost uniqueness of $\Psi$: if $\tilde \Psi:M_1 \to M_2$ is also such an extension of $\psi$, then $\dist_{M_2}(\Psi(x), \tilde \Psi(x))<\epsilon$ for any $x \in M_1$.
\end{theorem}

It is worth mentioning that for a map $\psi$ in the theorem above, we do not assume that $\psi$ is a $\delta$-Gromov-Hausdorff approximation to the image: namely, assuming (\ref{eq:heat comparison-intro}) is enough to conclude, see Lemma \ref{lem:surjective}.

\smallskip
Due to the convergence of the heat kernel, Problem A is equivalent to the rigidity problem, or the uniqueness of Gel'fand's inverse problem: {\it does the heat kernel on an open subset uniquely determine the geometry of a Ricci limit space?}
In this paper, we consider the rigidity problem in the more general synthetic ${\rm RCD}(K,N)$ framework, roughly speaking, metric-measure spaces with synthetic Riemannian Ricci curvature bounded below by $K$ and dimension bounded above by $N$.

The synthetic notion of Ricci curvature (lower) bound was introduced on general metric-measure spaces, called the ${\rm CD}(K,N)$ condition, in the seminal works \cite{LV,St1,St2} by Lott-Villani and Sturm independently.
The ${\rm CD}(K,N)$ condition defines a quite general class of spaces containing Riemannian and Finsler structures (in particular, all finite dimensional non-Euclidean normed spaces); however, at the same time, it has been considered to be too general for studying geometry.
To single out Riemannian-like behaviour, a more restrictive notion of synthetic \emph{Riemannian} Ricci curvature bound was introduced in \cite{AGS,G15}, called the ${\rm RCD}(K,N)$ spaces, by combining the linearity of the heat flow with the $\textrm{CD}$-condition.
The $\textrm{RCD}$-condition is stable under the measured Gromov-Hausdorff convergence, and contains (weighted) Riemannian manifolds, Ricci limit spaces and Alexandrov spaces \cite{Petrunin,ZZ,KMS}.
Let us mention that the closure of an open convex subset with measure-zero boundary of an $\RCD$ space is also an $\RCD$ space \cite[Prop. 7.7]{AMS}. In particular, this allows us to apply our results in the RCD setting to weighted Riemannian manifolds with non-smooth boundaries, which has independent interests from the study of inverse problems in the low regularity setting.

\smallskip
Gel'fand's inverse problem \cite{Ge} was originally formulated as the unique determination of a Schr{\"o}dinger operator from its boundary spectral data. On a bounded domain of $\mathbb{R}^n$, the problem was solved by \cite{NSU} based on the method of complex geometric optics developed in the seminal work \cite{SyU}, see \cite{Novikov88,RS88} for equivalent formulations.
Gel'fand's inverse problem has a natural geometric formulation on Riemannian manifolds: can the Riemannian metric be uniquely determined from the boundary spectral data of the Laplacian or a Schr{\"o}dinger operator? 
This spectral formulation is equivalent to the inverse problem for the heat kernel and an inverse problem for the wave equation from the hyperbolic Dirichlet-to-Neumann map \cite{KKL,KKLM}.
Typically the interior of the manifold itself is unknown so needs to be determined as part of the problem.
The uniqueness of Gel'fand's inverse problem was established for smooth compact Riemannian manifolds with boundary by \cite{BK,T95} in the 1990s, based on Tataru's unique continuation theorem and the Boundary Control method, introduced originally on $\mathbb{R}^n$ in \cite{Bel87}.
The unique solvability for Riemannian metrics of lower regularity $C^2_*$ was proved in \cite{AKKLT}.
Analogous formulations for closed Riemannian manifolds where the spectral data are measured in an open subset were studied in \cite{KrKL,HLOS}, see a recent relaxed formulation in \cite{FKU,FK} and a related formulation for the connection Laplacian \cite{KOP}.
Stronger H{\"o}lder-type stability estimates for the inverse problem were known under additional geometric assumptions \cite{SU98,SU05,AG}, for example when the metric is close to being simple (i.e., with convex boundary and no cut points), see \cite{BZ,BD,M14,SUV16} and the references therein.
In the general case, double-logarithmic stability estimates for Gel'fand's inverse problem were established in \cite{BKL,BILL,L25}, see also a recent related exponential instability result \cite{BLOS} based on the method developed in \cite{KRS}.
The uniqueness of Gel'fand's inverse problem for the non-smooth Gromov-Hausdorff limit spaces of Riemannian manifolds with bounded sectional curvature was proved in \cite{KLLY}, based on the $C^2_*$-regularity of the limit spaces of the orthonormal frame bundles within this class, and a stability estimate for the inverse problem was obtained for a class of orbifolds in \cite{LLY} using the stability method for local distance data developed in \cite{FILLN}.

Gel'fand's inverse problem can be formulated on general metric-measure spaces $(X,\dist,\meas)$, that is, a metric space $(X,\dist)$ equipped with a Borel measure $\meas$ of full support.
Assume that $(X,\dist,\meas)$ has synthetic Riemannian Ricci curvature bounded from below, i.e., satisfying the $\RCD(K,N)$ condition, where the heat kernel is well-defined and continuous.
The uniqueness of Gel'fand's inverse problem is open in the general metric-measure setting.

\smallskip
{\bf Problem B.} \emph{Let $(X,\dist,\meas)$ be a compact ${\rm RCD}(K,N)$ space and $V\subset X$ be an open subset.
Let $p(x,y,t)$ be the heat kernel of $(X,\dist,\meas)$, for $x,y\in X$ and $t\in \mathbb{R}_+$.
Does the heat kernel on $V\times V\times \mathbb{R_+}$ uniquely determine $(X,\dist,\meas)$ up to a measure-preserving isometry?}

\smallskip
Synthetic curvature appears in the study of mathematical inverse problems in the low regularity setting, motivated from imaging applications such as medical imaging and seismic imaging.
In medical imaging, the anisotropic conductivity is modelled as the Riemannian metric, and the celebrated Calder{\'o}n problem of recovering the conductivity is equivalent to determining the Riemannian metric, or its conformal class for dimension $2$, from boundary measurements \cite{LeeU,LaU}.
For another instance in seismic imaging, the physical material parameters of the solid Earth affect the wave speed that defines a Riemannian metric (more generally Finsler metric, see e.g. \cite{DIL}), and the latter may be reconstructed by solving Gel'fand's inverse problem for the elasticity system.
Curvature provides crucial information for solving geometric inverse problems, notably the well-known boundary rigidity problem \cite{SUV} and the Lorentzian Calder\'on problem \cite{AFO}. However, if physical parameters have lower regularity than $W^{2,p}$, for example, piecewise continuous as typically in practical imaging applications \cite{KLMS}, the curvature is not defined in the usual Riemannian sense, so studying the effect of curvature on inverse problems requires synthetic curvature.
Notably, a further advantage of studying inverse problems in the RCD framework is that treating domains with non-smooth boundaries is inherently more natural to the synthetic machinery for metric-measure spaces.

In the general setting, it was proved in \cite{BS} that an $\RCD(K,N)$ space, for some $K\in \mathbb{R}$ and some $N\in [1,\infty)$, has an essential dimension $n\in \mathbb{N}\cap [1,N]$ such that the set $\mathcal{R}_n$ of all $n$-regular points has full measure. Moreover, $\mathcal{R}_n$ is weakly convex \cite{D}. 
Recall that a set $\mathcal{R}\subset X$ is called weakly convex, if for any $x,y\in \mathcal{R}$ and any $\epsilon>0$, there exists an $\epsilon$-geodesic $\gamma\subset \mathcal{R}$ connecting $x,y$.
A piecewise geodesic curve $\gamma$ between $x,y$ is called an $\epsilon$-geodesic if its length $L(\gamma)\leq (1+\epsilon^2)\dist(x,y)$.
In this paper, we prove the uniqueness of the inverse problem on {\rm RCD}$(K,N)$ spaces if $\mathcal{R}_n$ has Lipschitz Riemannian structure.
Examples of such spaces include Lipschitz Riemannian manifolds with distributional Ricci curvature bounded below \cite{MR}, non-collapsed Ricci limit spaces with bounded Ricci curvature \cite{CC1}, and metric-measure spaces with mixed curvature bounds in the sense of \cite{KKK,HZ}. 
In particular, these include smooth Riemannian manifolds glued along common convex boundary \cite{Ketterer}, which can be applied to determining the interface (across which the metric is continuous but its normal derivative jumps) or Lipschitz anisotropic inclusions in the context of inverse obstacle problems.
In fact, we state the following slightly stronger result.

\begin{theorem}\label{thm:unique almost}
Let $(X,\dist,\meas)$ be a compact {\rm RCD}$(K,N)$ space for some $K\in \mathbb{R}$ and $1\leq N <\infty$.
Let $\mathcal{R}\subset X$ be an open set whose complement is $\meas$-null.
Assume that
\begin{itemize}
\item[(1)] $\mathcal{R}$ is an open weighted Lipschitz Riemannian manifold with locally Lipschitz density $\rho$;

\item[(2)] $\mathcal{R}$ is weakly convex.
\end{itemize}
Let $V\subset X$ be an open subset.
Then the heat kernel on $V\times V\times \mathbb{R_+}$ uniquely determine $(X,\dist,\meas)$ up to a measure-preserving isometry.
\end{theorem}

A more precise uniqueness formulation is stated in Theorem \ref{uniqueness-bijection}. 
Our proof is constructive in the sense that it reconstructs a metric-measure isometric copy of the original space, see Remark \ref{remark-reconstruction}.
Let us note that the open set $\mathcal{R}$ assumed in Theorem \ref{thm:unique almost} is a priori smaller than the set $\mathcal{R}_n$ of $n$-regular points, and that $\rho$ is not necessarily to be $C^1$ even if $\mathcal{R}$ is smooth, see for instance \cite{SZ} for an example of Ricci flat (collapsed) limits.

A primary example of Theorem \ref{thm:unique almost} is the non-collapsed Ricci limit spaces of Riemannian manifolds in $\mathcal{M}(n,D,v)$, since such limit spaces are well-known to have $C^{1,\alpha}$-Riemannian regular set \cite{CC1}.
We also mention that Theorem \ref{thm:unique almost} can be applied to some sub-Riemannian RCD spaces, see \cite{DHPW, P, PW, RS}. Note that we can easily construct an example of non-collapsed $\RCD(K,N)$ space satisfying the assumptions as in the theorem above, with no $C^{1, \alpha}$-regularity of the Riemannian metric by a graph of a function, see also \cite{KOV, MR} along this direction.

\begin{corollary}
    Theorem \ref{thm:unique almost} applies to non-collapsed Ricci limit spaces with bounded Ricci curvature, cross sections of tangent cones of such spaces,  Einstein orbifolds, and compact $\RCD(K, N)$ spaces whose metric structure is $\mathrm{CBA}(\kappa)$ for some $K \in \mathbb{R}$, $N \in [1, \infty)$ and $\kappa>0$. 
\end{corollary}
Recall that a metric space $(X, \dist)$ is $\mathrm{CAT}(\kappa)$ if the sectional curvature is bounded above by $\kappa$ in a synthetic sense and that it is $\mathrm{CBA}(\kappa)$ if it is locally $\mathrm{CAT}(\kappa)$. It is proved in \cite{KKK, HZ} that for any $\RCD(K,N)$ space $(X, \dist, \meas)$ with $\mathrm{CBA}(\kappa)$ condition, the boundary $\partial X$ is well-defined and $\meas$-null; moreover, the metric regular set $\mathcal{R}$ is a weighted Riemannian manifold with Lipschitz regularities requested in Theorem \ref{thm:unique almost}.
Our results can be also applied to weighted Riemannian manifolds with non-smooth convex boundary in the sense of \cite[Prop. 7.7]{AMS}.
Let us remark that some Einstein orbifolds constructed in \cite[Example 0.10]{Ozuch24} cannot be Ricci limits of non-collapsed smooth Einstein manifolds (see also \cite{O1, O2}), and that any Einstein orbifold is an RCD space (due to, for example, techniques in \cite{CW, H18, HS} or the local-to-global property \cite[Th. 7.8]{AMS}).

\smallskip
Theorem \ref{thm:unique almost} extends the uniqueness result \cite{KLLY} on Gel'fand's inverse problem for collapsing Riemannian manifolds with bounded sectional curvature to the synthetic setting, and improves the method so that it does not depend on the smooth approximating sequences of manifolds.
As a special case, our results are valid for Lipschitz Riemannian manifolds with synthetic Riemannian Ricci curvature bounded from below, lowering the minimal regularity for solving the inverse problem from $C^2_*$ to Lipschitz.
Recall that $C^{1,1}\subset C^2_* \subset C^{1,\alpha}$ for any $0<\alpha<1$, see e.g. \cite[Ch. 2.7]{Triebel}.
The relaxation to Lipschitz regularity is essential because the usual Riemannian techniques do not work (e.g., locally geodesic equation is not uniquely solved for $C^{1,\alpha}$-metric \cite{Hartman}), and we rely on metric geometry from RCD spaces to make this possible. 

Our proof is based on Tataru-type unique continuation for wave operators and a non-smooth version of the Boundary Control method.
The unique continuation for wave operators requires essentially the minimal Lipschitz metric regularity, due to counterexamples for operators with H{\"o}lder second-order coefficients \cite{Mandache}.
We also refer readers to the interesting counterexample to the (strong) unique continuation for harmonic functions in the general RCD setting \cite{DZ}.
Under Lipschitz regularity, the geodesic equation is locally not uniquely solvable (while uniquely solvable for $C^2_*$-metric, see e.g. \cite[Prop. A.2]{Taylor-book}) and the classical method does not work.
Instead, geodesics are understood in the sense of metric geometry and are non-branching under the RCD condition \cite{D}, see a brief review on RCD spaces in Section \ref{sec-RCD}.
The basic idea is to use the unique continuation on the Lipschitz Riemannian regular set, and perform a non-smooth version of the  Boundary Control method to identify the eigenfunctions and the heat kernel on the whole space.
To establish the unique continuation for wave operators in our setting in Section \ref{sec-uc}, we prove the well-posedness and finite speed of propagation for the wave equation on compact RCD spaces in Section \ref{sec-wave}.
Then using Varadhan-type asymptotics and short-time diagonal asymptotics of the heat kernel in the general RCD setting, we determine the distance and measure to prove Theorem \ref{thm:unique almost} in Section \ref{sec-IP}.
Finally, we prove Theorem \ref{thm:function-intro} in Section \ref{sec-stability}.
We hope that the non-smooth methods developed in this paper can be used to study other types of inverse problems.

\smallskip
\noindent {\bf Acknowledgements.}
The authors would like to thank Nicola Gigli for his interest in this project and for helpful discussions and bibliographic suggestions.
The second named author would like to thank Spyridon Filippas for helpful discussions on the unique continuation under Lipschitz regularity.
The second named author is grateful to Matti Lassas for introducing him to the field of inverse problems and to the inverse problems group at University of Helsinki for stimulating environment.
The first named author acknowledges support of the Grant-in-Aid for Scientific Research (A) of 25H00586, the Grant-in-Aid for Scientific
Research (B) of 20H01799, the Grant-in-Aid for Scientific Research (B) of 21H00977
and Grant-in-Aid for Transformative Research Areas (A) of 22H05105. 
The second named author acknowledges support for this work from the Research Council
of Finland, the Finnish Centre of Excellence of Inverse Modelling and Imaging grant 352948 and FAME Flagship grant 359182.

\section{RCD spaces} \label{sec-RCD}

\subsection{Definition}
In this subsection let us provide a quick introduction on RCD spaces. We refer to \cite{Gigli_survey, GP} for details.

In the sequel, we fix a metric measure space, $(X, \dist_X, \meas_X)$ (or $(X, \dist, \meas)$, or $X$ for short), namely $(X, \dist)$ is a complete separable metric space, and $\meas$ is a Borel measure which is finite and positive on any open ball $B_r(x)$ of radius $r>0$ centered at $x$ (thus $\supp \meas =X$). The Sobolev space $H^{1,2}(X)$ is defined by the finiteness domain of the Cheeger energy $\Ch:L^2(X, \meas) \to [0, \infty]$;
\begin{equation}
    \Ch(f)=\inf_{f_i}\left\{ \liminf_{i\to \infty} \frac{1}{2}\int_X(\mathrm{Lip}f_i)^2\di\meas\right\},
\end{equation}
where the infimum runs over all bounded Lipschitz functions $f_i$ with $f_i \to f$ in $L^2$, and $\mathrm{Lip}f_i$ denotes the local slope or local Lipschitz constant. It is a Banach space endowed with the norm $\|f\|_{H^{1,2}}=(\|f\|_{L^2}^2+2\Ch(f))^{\frac{1}{2}}$. For any $f \in H^{1,2}(X)$, there exists a \textit{canonical} object $|\nabla f| \in L^2(X)$ such that $\Ch(f)=\frac{1}{2}\int_X|\nabla f|^2 \di \meas$ holds. Then the Sobolev-to-Lipschitz property of $X$ means that for any $f \in H^{1,2}(X)$ with $|\nabla f| \le 1$ for $\meas$-a.e., $f$ has a $1$-Lipschitz representative. Note that the local notions of the above including $H^{1,2}(U)$ (more generally $H^{1,p}(U)$ for $p \in [1, \infty]$) and $D(\Delta, U)$ for an open subset $U$ of $X$ are well-defined, where we immediately use them in the sequel. See for instance \cite{BB} for the details. It is also worth mentioning that for a wide class of metric measure spaces (including $\RCD(K, N)$ spaces for $N<\infty$), $|\nabla f|(x)=\mathrm{Lip}f(x)$ holds for $\meas$-a.e. $x \in U$ if $f$ is locally Lipschitz \cite{Ch99}.

We say that $X$ is infinitesimally Hilbertian if $H^{1,2}(X)$ is a Hilbert space. In the sequel, we assume that $X$ is infinitesimally Hilbertian, then for all $f_i \in H^{1,2}(X)(i=1,2)$, we can define the inner product of the gradients by
\begin{equation}
    \langle \nabla f_1, \nabla f_2\rangle =\lim_{\epsilon \to 0}\frac{|\nabla (f_1+\epsilon f_2)|^2-|\nabla f_1|^2}{2\epsilon} \in L^1(X).
\end{equation}
Let us define the domain $D(\Delta)$ of the Laplacian by the set of all $f \in H^{1,2}(X)$ satisfying that there exists a unique $\phi \in L^2$, denoted by $\Delta_Xf$, or by $\Delta f$ for short, such that
\begin{equation}\label{eq:laplacian definition}
    \int_X\langle \nabla f, \nabla g\rangle \di\meas=-\int_X\Delta f \cdot g \di\meas, \quad \text{for any $g \in H^{1,2}(X)$.}
\end{equation}
We are now in a position to introduce the definition of RCD spaces, where the following volume growth condition for some $C \ge 1$ and some $x \in X$;
\begin{equation}\label{eq:volume}
        \meas(B_r(x)) \le e^{Cr^2},\quad \text{for any $r \ge 1$,}
    \end{equation}
    will play a role to guarantee the stochastic completeness.
\begin{definition}[RCD space]
    A metric measure space $(X, \dist, \meas)$ is said to be an $\RCD(K, N)$ space for some $K \in \mathbb{R}$ and some $N \in [1, \infty]$, or an RCD space for short, if it is infinitesimally Hilbertian, it satisfies the Sobolev-to-Lipschitz property, (\ref{eq:volume}) holds, and the Bochner inequality is satisfied in the following sense,
    \begin{equation}\label{eq:bochner ineq}
        \frac{1}{2}\int_X\Delta \phi |\nabla f|^2\di\meas \ge \int_X\phi\left( \frac{(\Delta f)^2}{N}+\langle \nabla \Delta f, \nabla f\rangle +K|\nabla f|^2\right)\di \meas,
    \end{equation}
    for all $f \in D(\Delta)$ with $\Delta f \in H^{1,2}(X)$, and $\phi \in D(\Delta) \cap L^{\infty}(X)$ with $\phi \ge 0$ and $\Delta \phi \in L^{\infty}(X)$. 
\end{definition}

Note that the regularity of $f, \phi$ requested in the definition above is, a priori, optimal to justify (\ref{eq:bochner ineq}). For a broad audience, let us recall the validity of the definition in the simplest case: for a complete $n$-dimensional Riemannian manifold $M$ without boundary, $M$ is an $\RCD(K,N)$ space if and only if $n \le N$ and $\mathrm{Ric}(M) \ge K$ hold. Let us provide a sketch of the proof of this characterization for the convenience of readers as follows (see \cite{EKS, Han} for more general results).

Recall the Bochner identity for any smooth function $f$ on $M$,
\begin{equation}\label{eq:pointwise Bochner}
    \frac{1}{2}\Delta|\nabla f|^2=|\mathrm{Hess}(f)|^2+\langle \nabla f, \nabla \Delta f\rangle+\mathrm{Ric}(\nabla f, \nabla f),
\end{equation}
and an elementary inequality $|\mathrm{Hess}(f)|^2 \ge \frac{(\Delta f)^2}{n}$. 
First let us prove ``if'' part. Thanks to the Bochner identity, we can prove (\ref{eq:bochner ineq}) for smooth functions $f$ if integration by parts works. Then we can justify this for general $f$ by using the heat flow, explained in the next subsection. Next let us prove ``only if'' part. Notice that (\ref{eq:bochner ineq}) together with cut-off yields for any smooth function $f$,
\begin{equation}\label{eq:pointwise Bochner ineq}
    \frac{1}{2}\Delta|\nabla f|^2 \ge \frac{(\Delta f)^2}{N}+\langle \nabla f, \nabla \Delta f\rangle+K|\nabla f|^2.
\end{equation}
Finding  a harmonic chart around any point $x$ with vanishing Hessians at $x$, it follows from applying the chart to (\ref{eq:pointwise Bochner}) and (\ref{eq:pointwise Bochner ineq}) that $\mathrm{Ric}(M)\ge K$ holds. The remaining inequality $n \le N$ is justified by the Bishop-Gromov inequality for RCD spaces because this implies that the Hausdorff dimension is at most $N$.

\smallskip
Finally let us introduce a special subclass of RCD spaces, called \textit{non-collapsed} spaces, which will appear in our applications.
\begin{definition}[Non-collapsed RCD space \cite{DG}]
    An $\RCD(K, N)$ space $(X, \dist, \meas)$ is said to be \textit{non-collapsed} if $\meas$ coincides with the $N$-dimensional Hausdorff measure $\mathcal{H}^N$.
\end{definition}

\subsection{Heat kernel}
Let us fix an $\RCD(K, N)$ space $X$ for some $K \in \mathbb{R}$ and some $N \in [1, \infty)$, where some of the following are also justified (as possibly weaker formulations) in more general setting (e.g. PI spaces).

The gradient flow of the Cheeger energy is called the heat flow, denoted by $h_t:L^2(X) \to L^2(X)$, namely for any $f \in L^2(X)$, there exists a unique absolutely continuous (eventually smooth, see \cite{GP}) curve $h_{\cdot}f:(0, \infty) \to L^2(X)$ such that $h_tf \in D(\Delta)$ holds for any $t \in \mathbb{R}_+$, that $h_tf \to f$ holds in $L^2$ as $t \to 0^+$, and that $\frac{\di}{\di t}h_tf=\Delta h_tf$ holds for $\mathcal{L}^1$-a.e. $t \in \mathbb{R}_+$ (thus eventually for any $t \in \mathbb{R}_+$, by the smoothness). The heat flow can be written by the canonical symmetric continuous function, called the heat kernel, $p=p_X:X \times X \times \mathbb{R}_+ \to \mathbb{R}$, in the sense: for any $f \in L^2(X)$,
\begin{equation}
    h_tf(x)=\int_Xf(y)p(x,y,t)\di\meas(y), \quad \text{for $\meas$-a.e. $x \in X$.}
\end{equation}
Let us recall Gaussian estimates on $p$ established in \cite{JLZ}.
\begin{theorem}[Gaussian estimate]\label{thm:gaussian}
For any $\epsilon>0$, there exists $C=C(K, N, \epsilon)>1$ such that 
\begin{equation} \label{eq-Gaussian-estimate}
    \frac{1}{C\meas(B_{\sqrt{t}}(x))}\cdot\exp \left(-\frac{\dist(x,y)^2}{4(1-\epsilon)t} -Ct\right)\le p(x, y, t) \le \frac{C}{\meas(B_{\sqrt{t}}(x))}\cdot\exp \left(-\frac{\dist(x,y)^2}{4(1+\epsilon)t} +Ct\right)
\end{equation}
holds for all $t\in \mathbb{R}_+$ and all $x, y \in X$. Moreover,
\begin{equation}\label{eq-gradient-Gaussian-estimate}
    |\nabla_xp(x, y, t)|\le \frac{C}{\sqrt{t}\meas(B_{\sqrt{t}}(x))}\cdot \exp \left(-\frac{\dist(x,y)^2}{4(1+\epsilon)t}+Ct\right)
\end{equation}
    holds for all $t \in \mathbb{R}_+$ and $\meas$-a.e. $x,y \in X$.
\end{theorem}


A related rigidity result with Euclidean heat kernel was recently studied in \cite{CT22}.
The Varadhan-type asymptotics for the heat kernel is a direct consequence of the Gaussian estimate \eqref{eq-Gaussian-estimate} and the Gromov-Bishop volume growth (e.g. \cite[Th. 2.3]{St2}).

\begin{proposition}[Varadhan's asymptotics]\label{prop:varadhan}
    The Varadhan-type asymptotics for the heat kernel holds for $\RCD(K,N)$ spaces:
    $$-\lim_{t\to 0^+} 4t \log p(x,y,t) = \dist(x,y)^2,\quad \textrm{ for any }x,y\in X. $$
\end{proposition}
Finally, let us mention that the heat kernels are continuous with respect to the pointed measured Gromov-Hausdorff convergence $X_i \to X$ of pointed $\RCD(K, N)$ spaces (see \cite{AHT}):
\begin{equation}\label{eq:heat conv}
    p_{X_i}(x_i, y_i, t_i) \to p_X(x, y, t), \quad \text{whenever $x_i, y_i \in X_i \to x, y \in X$, resp. and $t_i \to t$ in $\mathbb{R}_+$.}
\end{equation}


\subsection{Eigenfunctions and spectral datum}
Let $(X, \dist, \meas)$ be an $\RCD(K, N)$ space for some $K \in \mathbb{R}$ and some $N \in [1, \infty)$. Since it is known that $X$ is a geodesic PI space (see for instance \cite{BB}), for all $x \in X$ and $r>0$, the canonical inclusion $H^{1,2}(B_r(x)) \hookrightarrow L^2(B_r(x))$ is a compact operator.  In particular, if $X$ is compact, then the standard functional analysis allows us to write the spectrum of the Laplacian on $X$, counted with their multiplicities,  as
\begin{equation}
    0=\lambda_0<\lambda_1\le \lambda_2 \le \cdots \to \infty,
\end{equation}
where $f \in L^2(X)$ is said to be an eigenfunction of the eigenvalue $\lambda$ if $f \in D(\Delta)$ with $\|f\|_{L^2}\neq 0$ and $\Delta f+\lambda f=0$. Fixing a \textit{spectral datum} $\{\phi_i\}_{i\in \mathbb{N}}$ of $X$ (namely, it is a complete orthonormal basis of $L^2(X)$ with $\Delta \phi_i+\lambda_i\phi_i=0$, where we will also use the detailed notations $\lambda_i^X, \phi_i^X$), 
thanks to (\ref{eq-Gaussian-estimate}) and (\ref{eq-gradient-Gaussian-estimate}), we have the following gradient-$L^{\infty}$-estimates and  a Weyl-type bound (e.g. \cite{AHPT});
\begin{equation}\label{eq:eigenfunction-estimate}
    \|\nabla \phi_i\|_{L^{\infty}} \le C\lambda_i^{\frac{N+2}{4}},\quad \|\phi_i\|_{L^{\infty}} \le C\lambda_i^{\frac{N}{4}}, \quad \lambda_i \ge C^{-1}i^{\frac{2}{N}},
\end{equation}
where $C$ is a positive constant depending only on $K, N$ and an upper bound on the diameter $\mathrm{diam}(X)$ of $X$. In particular, we know
\begin{equation}\label{eq:heat exp}
    p(x, y, t)=\sum_{i \ge 0}e^{-\lambda_it}\phi_i(x)\phi_i(y), \quad \text{in $C(X \times X)$}
\end{equation}
for any $t>0$, and 
\begin{equation}
    p(\cdot, y, t)=\sum_{i \ge 0}e^{-\lambda_it}\phi_i(y)\phi_i,\quad \text{in $H^{1,p}(X)$}
\end{equation}
for all $y \in X$, $t \in \mathbb{R}_+$ and $p \in [1, \infty)$.

\smallskip
Let us mention that the local versions of the following are main targets of the paper.
\begin{theorem}\label{thm:global rigidity}
    Let $X, Y$ be compact $\RCD(K,N)$ spaces for some $K \in \mathbb{R}$ and some $N \in [1, \infty)$, and let $\psi:X \to Y$ be a map preserving the heat kernels. Then $\psi$ is a metric measure isometry.
\end{theorem}
\begin{proof}
    Since
    \begin{equation}
        \sum_{i \ge 0}e^{-\lambda^X_it}\phi_i^X(x)\phi_i^X(\tilde x)=\sum_{i \ge 0}e^{-\lambda_i^Yt}\phi_i^Y(\psi(x))\phi_i^Y(\psi(\tilde x)),
    \end{equation}
    letting $t \to \infty$ shows $\meas_X(X)=\meas_Y(Y)$. Integrating this over $\tilde x \in X$ yields
    \begin{equation}
        0=\sum_{i \ge 1}e^{-\lambda_i^Yt}\phi_i^Y(\psi(x))\int_Y\phi_i^Y\di \psi_{\sharp}\meas_X.
    \end{equation}
    Integrating this over $x \in X$ shows $\int_Y\phi_i^Y\di \psi_{\sharp}\meas_X=0$ for any $i \ge 1$. Then, applying \cite[Lem. 2.3]{H24}, we know $\psi_{\sharp}\meas_X=\meas_Y$. In particular $\psi$ is surjective. On the other hand, Varadhan's asymptotics yields that $\psi$ preserves the distances. Thus we conclude.
\end{proof}
Let us use a standard notation in the convergence theory:
\begin{equation}\label{eq:small}
\varepsilon_{c_1, c_2, \ldots, c_m} (\epsilon_1, \epsilon_2, \ldots, \epsilon_l)=\varepsilon (\epsilon_1, \epsilon_2, \ldots, \epsilon_l| c_1, \ldots, c_m)
\end{equation}
denotes a function $\varepsilon: (\mathbb{R}_{>0})^l \times \mathbb{R}^m \to (0, \infty)$ satisfying
\begin{equation}
\lim_{(\epsilon_1, \ldots, \epsilon_l) \to 0}   \varepsilon_{c_1, c_2, \ldots, c_m} (\epsilon_1, \epsilon_2, \ldots, \epsilon_l)=0, \quad \text{for all $c_i \in \mathbb{R}$.}
\end{equation}
\begin{corollary}\label{cor:global rigidity}
    Let $X, Y$ be compact $\RCD(K,N)$ spaces for some $K \in \mathbb{R}$ and some $N \in [1, \infty)$ with $\mathrm{diam}(X) \le D$, $\mathrm{diam}(Y)\le D$, $D^{-1}\le \meas_X(X)\le D$ and $D^{-1}\le \meas_Y(Y) \le D$ for some $D \ge 1$, and let $\psi:X \to Y$ be a map satisfying
    \begin{equation}
        1-\epsilon \le \frac{p_X(x, \tilde x, t)}{p_Y(\psi(x), \psi(\tilde x), t)}\le 1+\epsilon,\quad \text{for all $x, \tilde x \in X$ and $t \in [\epsilon ,1]$}.
    \end{equation}
    Then $\psi$ is an $\varepsilon_{K, N, D}(\epsilon)$-mGH approximation.
\end{corollary}
\begin{proof}
    Though this is done by a standard contradiction argument together with the continuity of the heat kernels with respect to mGH convergence (\ref{eq:heat conv}) and Theorem \ref{thm:global rigidity}, let us provide a sketch of the proof for readers' convenience.

    Assume that the assertion is not satisfied. Then there exists sequences of compact $\RCD(K,N)$ spaces $X_i, Y_i$ with $\mathrm{diam}(X_i)\le D$ and $\mathrm{diam}(Y_i)\le D$, $\tau>0$ and a sequence of maps $\psi_i:X_i \to Y_i$ such that $\psi_i$ is not a $\tau$-mGH approximation, and that 
    \begin{equation}
        1-\epsilon_i \le \frac{p_{X_i}(x, \tilde x, t)}{p_{Y_i}(\psi_i(x), \psi_i(\tilde x), t)}\le 1+\epsilon_i,\quad \text{for all $i \in \mathbb{N}$, $x, \tilde x \in X_i$ and $t \in [\epsilon_i ,1]$,}
    \end{equation}
    for some $\epsilon_i \to 0$. Thanks to the compactness of $\RCD(K, N)$ spaces with respect to mGH convergence (e.g. \cite{AGS2, EKS, GMS, LV, St1}, see also the end of the next subsection), after passing to a subsequence, with no loss of generality, we can assume $X_i \stackrel{\mathrm{mGH}}{\to} X, Y_i \stackrel{\mathrm{mGH}}{\to} Y$ for some compact $\RCD(K, N)$ spaces $X, Y$.
    Applying Theorem \ref{thm:gaussian}, after passing to a subsequence again, we can find a uniform limit map $\psi:X \to Y$ of $\psi_i$ preserving the heat kernels (see step 2 of the proof of (1) of Proposition \ref{prop:weak}). Thus Theorem \ref{thm:global rigidity} shows that $\psi$ is a metric measure isometry. In particular $\psi_i$ is $\delta_i$-mGH approximation for some $\delta_i \to 0$, which is a contradiction, where we note that the above is also justified in the case when $X$ or $Y$ is a single point (see also \cite[Sec. 6]{H24} and the proof of Lemma \ref{lem:surjective}).
\end{proof}

\subsection{Structure theory}\label{subsec:structure}
Let $(X, \dist_X, \meas_X)$ be an $\RCD(K, N)$ space for some $K \in \mathbb{R}$ and some $N \in [1, \infty)$. A pointed complete metric space $(Y, \dist_Y, y)$ is said to be a metric tangent cone at $x \in X$ of $X$ if there exists $r_i \to 0^+$ such that $(X, r_i^{-1}\dist_X, x)$ pointed Gromov-Hausdorff (pGH) converge to $(Y, \dist_Y, y)$. Moreover a (measured) tangent cone $(Y, \dist_Y, \meas_Y, y)$ is similarly defined by satisfying $(X, r_i^{-1}\dist_X, \meas_X(B_{r_i}(x))^{-1}\meas_X, x) \to (Y, \dist_Y, \meas_Y, y)$ in the pointed measured Gromov-Hausdorff (pmGH) convergence, where pGH and pmGH topologies are metrizable by $\dist_{\mathrm{pGH}}$ and $\dist_{\mathrm{pmGH}}$, resp.. Denote by $\mathrm{Tan}_{\mathrm{met}}(X, x)$ (or $\mathrm{Tan}(X, x)$, resp.) the set of all isometry classes of tangent cones (or metric tangent cones, resp.) at $x \in X$.
It is worth mentioning that $\mathrm{Tan}_{\mathrm{met}}(X, x)\neq \emptyset$, $\mathrm{Tan}_{\mathrm{met}}(X, x)\neq \emptyset$, and any element of $\mathrm{Tan}_{\mathrm{met}}(X, x)$ is an $\RCD(0, N)$ space, thus we can apply the splitting theorem \cite{G13} to all tangent cones. In particular, if the metric tangent cone at $x$ is unique as $\mathbb{R}^n$, in other words, $\mathrm{Tan}_{\mathrm{met}}(X, x)=\{(\mathbb{R}^n, 0_n)\}$, then the uniqueness is also valid in the sense of the measured ones;
\begin{equation}\label{eq:regular}
  \mathrm{Tan}_{\mathrm{met}}(X, x)=\{(\mathbb{R}^n, \omega_n^{-1}\mathcal{H}^n, 0_n)\},  
\end{equation}
where we note that the converse implication is trivial.
A point $x \in X$ is said to be $n$-regular if (\ref{eq:regular}) holds. Denote by $\mathcal{R}_{X, n}$, or $\mathcal{R}_n$ for short, the set of all $n$-regular points of $X$. 

We are ready to introduce general local structure results of $X$ (see \cite{AHT, BS, D}).
\begin{theorem}[General structure result]\label{thm:rcd structure}
    Assume that $X$ is not a single point. Then we have the following.
    \begin{enumerate}
    \item{(Non-branching)} $X$ is proper geodesic space, and any geodesic cannot be branched at any interior point \cite{D}. Furthermore, defining the \textit{cut locus} at $x \in X$ by
\begin{equation}
    C_x=\{y \in X\,|\, \dist(x, y)+\dist(y,z)>\dist(x, z)\, \text{for any $z \neq y$}\},
\end{equation}
then we have $\meas(C_x)=0$. Moreover if $r_x=r_y$ for some $x, y \in X$, where $r_x$ is the distance function from $x$, on a Borel subset $A$ of $X$ with $\meas(A)>0$, then $x=y$ \cite{Ye}.
    \item{(Dimension)} There exists a unique $n \in \mathbb{N} \cap [1, N]$, called the essential (or rectifiable) dimension of $X$, denoted by $\mathrm{dim}(X)$, such that $\meas(X \setminus \mathcal{R}_n)=0$ holds \cite{BS}.
        \item{(Convexity)} $\mathcal{R}_n$ is weakly convex \cite{D}. 
        \item{(Rectifiability)} $X$ is strongly (metric measure) $n$-rectifiable \cite{MN} with \cite{DMR, GPe, KM}.
        \item{(Relation with Hausdorff measure)} Denoting by $\mathcal{R}_n^*$, called the $n$-dimensional reduced regular set, the set of all points $x \in \mathcal{R}_n$ satisfying that the following limit exists as a positive finite number;
        \begin{equation}\label{eq:radon}
            \lim_{r \to 0^+}\frac{\meas(B_r(x))}{\omega_nr^n}\in \mathbb{R}_+,
        \end{equation}
        one has $\meas (X \setminus \mathcal{R}_n^*)=0$. Furthermore, $\meas$ and $\mathcal{H}^n$ are mutually absolutely continuous on $\mathcal{R}_n^*$, and the limit (\ref{eq:radon}) coincides with the Radon-Nikodym derivative $\frac{\di\meas}{\di\mathcal{H}^n}$ on $\mathcal{R}_n^*$ \cite{AHT}.
    \end{enumerate}
\end{theorem}
It is worth mentioning that in general $\dim (X)$ does not coincides with the Hausdorff dimension \cite{PW} (see also \cite{DHPW}), and that $t^{\frac{n}{2}}p(x,x, t)$ converge to the inverse of (\ref{eq:radon}), up to a multiplication of a positive constant depending only on $n$ for any $x \in \mathcal{R}_n^*$ (see (\ref{short-time-diagonal}) and \cite{AHT}), and that $X$ is non-collapsed if and only if $\mathrm{dim}(X)=N$, up to a multiplication of a positive constant to the reference measure $\meas$, see \cite{BGHZ, DG,H20}.

Finally let us end this subsection by noticing compactness results for $\RCD(K, N)$ spaces (though we already used this in the proof of Corollary \ref{cor:global rigidity}). If a sequence of pointed $\RCD(K, N)$ spaces has uniform two-sided positive bounds on the volumes of the balls of radius $1$ centered at the base points, then the sequence has a pmGH convergent subsequence to a pointed $\RCD(K, N)$ spaces, where an analogue of non-collapsed ones also exists, see \cite{DG, EKS, GMS} for the details.

\section{Wave equation on RCD spaces} \label{sec-wave}

Let $(X,\dist,\meas)$ be a compact $\RCD(K,N)$ space for some $K \in \mathbb{R}$ and some $N \in [1, \infty)$.
We consider the existence and uniqueness of solutions to the following wave equation:
\begin{equation} 
\label{eq-wave-intro}
\left\{ \begin{aligned}
&(\partial_t^2-\Delta_X) u(x,t)=f(x,t), \quad x\in X, \; t\in [0,T] ,\\
& u(x,0)= \psi_0(x), \\
& \partial_t u(x,0)=\psi_1(x),
\end{aligned} \right.
\end{equation}
where $\Delta_X$ is the Laplacian of $(X,\dist,\meas)$ defined in \eqref{eq:laplacian definition}.
Let us recall that the Laplacian $\Delta_X$ corresponds to the Neumann Laplacian in the special case of smooth Riemannian manifolds with (convex) boundary.
First, we define a weak formulation of the wave equation.
Recall that $L^{\infty}([0,T];H^{1,2}(X))$ is the space of functions $u:[0,T]\to H^{1,2}(X)$ satisfying ${\rm ess \,sup}_{t\in [0,T]} \|u(t)\|_{H^{1,2}(X)}<\infty$.
The space ${\rm Lip} ([0,T]; L^{2}(X))$ is the space of functions such that $u:[0,T]\to L^{2}(X)$ is Lipschitz.
For the weak formulation, we consider the solution in the regularity class
\begin{equation}
E([0,T];X):= L^{\infty}([0,T]; H^{1,2}(X)) \cap {\rm Lip} ([0,T]; L^2(X)).
\end{equation}
As $L^2(X,\meas)$ is reflexive so satisfies the Radon-Nikodym property (e.g. \cite[Cor. 1.2.7]{ABHN}), any function in ${\rm Lip} ([0,T]; L^{2}(X))$ is absolutely continuous and differentiable a.e. in the strong sense, and in particular, its derivative coincides with the weak derivative.

\begin{definition}[Weak formulation] \label{weak-formulation}
For $\psi_0\in H^{1,2}(X)$, $\psi_1\in L^2(X)$, $f\in L^1([0,T];L^2(X))$, we say that a function
\begin{equation} \label{energy-class}
u\in E([0,T];X)
\end{equation}
is a weak solution of the wave equation \eqref{eq-wave-intro}, if $u|_{t=0}=\psi_0$ and
\begin{equation} \label{def-weak}
\int_0^T \int_X \Big(-\partial_t u\, \partial_t v+ \langle \nabla u,\nabla v\rangle \Big) \di\meas \di t= \int_0^T \int_X f v +\int_X v|_{t=0} \psi_1,
\end{equation}
for any test function $v\in E([0,T];X)$ satisfying $v|_{t=T}=0$.
\end{definition}

Note that $u\in E([0,T];X)$ implies that $u\in C([0,T];L^{2}(X))$ which justifies the condition $u|_{t=0}=\psi_0$, and $u$ is weakly differentiable with respect to $t$ with $\partial_t u \in L^{\infty}([0,T];L^2(X))$.

\begin{proposition} \label{prop-regularity}
Let $(X,\dist,\meas)$ be a compact {\rm RCD}$(K,N)$ space for some $K \in \mathbb{R}$ and $N \in [1, \infty)$.
Let $T>0$, $f \in C([0,T]; L^2(X))$, $\psi_0\in  H^{1,2}(X)$ and $\psi_1\in L^2(X)$. Then there exists a unique weak solution
$$u\in L^{\infty}([0,T]; H^{1,2}(X)) \cap {\rm Lip} ([0,T]; L^2(X)),$$
with the energy estimate
$$\sup_{0\leq t \leq T} \|u(t)\|_{H^{1,2}} + \sup_{0\leq t,s \leq T,\, t\neq s} \frac{\|u(t)-u(s)\|_{L^2}}{|t-s|}  \leq C(T) \Big( \|f\|_{L^2(X\times [0,T])} + \|\psi_0\|_{H^{1,2}(X)} + \|\psi_1\|_{L^2(X)} \Big).$$
\end{proposition}

\begin{proof}
We mainly follow \cite[Sec. 2.3]{KKL} to prove the existence of weak solutions through Galerkin approximations. For $j\in \mathbb{N}$, let $\lambda_j \geq 0$ be eigenvalues of $-\Delta_X$, and $\phi_j$ be the corresponding eigenfunctions that form an orthonormal basis of $L^2(X,\meas)$. In particular, $\lambda_0=0$ and $\phi_0=\textrm{const}$.
For $\psi_0\in H^{1,2}(X)$, we take its Fourier expansion with respect to the basis,
\begin{equation}
\psi_0(x)=\sum_{j=0}^{\infty} \psi_{0,j} \phi_j(x), 
\end{equation}
where $\psi_{0,j}=\int_X \psi_0 \phi_j \di\meas$.
We consider functions of the form
\begin{equation}\label{def-u}
u(x,t)=\sum_{j=0}^{\infty} u_j(t) \phi_j(x),
\end{equation}
with
\begin{equation} \label{eq-uj}
u_j(t) =\psi_{0,j} s_j'(t)+\psi_{1,j} s_j(t) + \int_0^t f_j(\tau) s_j(t-\tau) \di\tau,
\end{equation} 
where $\psi_{0,j},\psi_{1,j},f_j(\tau)$ are the Fourier coefficients of $\psi_0,\psi_1,f(\tau)$ in the wave equation \eqref{eq-wave-intro} by taking $L^2$-inner product with $\phi_j$, and
\begin{equation} \label{eq-sj}
s_j(t) = \left\{ \begin{aligned}
&\frac{\sin(\sqrt{\lambda_j} t)}{\sqrt{\lambda_j}}, \quad & j\geq 1, \\
& t, \quad & j=0.
\end{aligned} \right.
\end{equation}
We will show that the function $u$ in \eqref{def-u} is well-defined in $L^{\infty}([0,T]; H^{1,2}(X)) \cap {\rm Lip} ([0,T]; L^2(X))$, and is a weak solution of the wave equation \eqref{eq-wave-intro}.
The idea is to consider the function
\begin{equation}
u^J(x,t):=\sum_{j=0}^J u_j(t) \phi_j(x),
\end{equation}
and use it to approximate the function $u$ as $J\to \infty$.

\medskip
{\bf Step 1.} We show that $u^J \in L^{\infty}([0,T]; H^{1,2}(X)) \cap {\rm Lip} ([0,T]; L^2(X))$ satisfies the desired energy estimate, and $u^J$ satisfies the integral identity \eqref{def-weak}.

\smallskip
First, we verify that $u^J$ is in the energy class.
Since $\phi_j\in D(\Delta_X)$ form an orthonormal basis of $L^2(X,\meas)$, 
by the definition and linearity of $\Delta_X$, for any fixed $t\in [0,T]$,
\begin{eqnarray} \label{eq-H1N}
\|u^J(t)\|^2_{H^{1,2}(X)}&=& \|u^J(t)\|^2_{L^2} + \int_X \langle \nabla u^J(t), \nabla u^J(t) \rangle \di\meas \nonumber \\
&=& \sum_{j=1}^{J} u^2_j - \int_X u^J \Delta_X u^J \di\meas = \sum_{j=1}^J (1+\lambda_j) u_j^2(t).
\end{eqnarray}
Using the specific form of $u_j$ in \eqref{eq-uj}, and $|s_j(t)|\leq t \leq T$, $|s'_j(t)|\leq 1$, and $\lambda_j s_j^2 \leq 1$ for $j\geq 1$, we have
\begin{eqnarray} \label{eq-H1N-estimate}
\|u^J(t)\|^2_{H^{1,2}(X)} &\leq & C(T) \sum_{j=0}^J (1+\lambda_j) \Big(\psi_{0,j}^2 + \psi_{1,j}^2 s_j^2 +\int_0^t f^2_{j}(\tau) s_j^2(t-\tau) \di\tau \Big) \nonumber \\
&=& C(T) \sum_{j=0}^J \Big(\psi_{0,j}^2 + \psi_{1,j}^2 s_j^2 +\int_0^t f^2_{j}(\tau) s_j^2 \di\tau \Big)  + C(T)\sum_{j=1}^J \lambda_j \Big(\psi_{0,j}^2 + \psi_{1,j}^2 s_j^2 +\int_0^t f^2_{j}(\tau) s_j^2 \di\tau \Big) \nonumber \\
&\leq & C(T) \Big( \|\psi_0\|^2_{H^{1,2}(X)} +  \|\psi_1\|^2_{L^2(X)} + \|f\|^2_{L^2(X\times [0,T])} \Big),
\end{eqnarray}
where we have used that, see e.g. \cite[Prop. 4.5]{H18},
\begin{equation} \label{def-H1-norm}
\|\psi_0\|^2_{H^{1,2}(X)}=\sum_{j=0}^{\infty} (1+\lambda_j) \psi_{0,j}^2.
\end{equation}

For any $t_1,t_2 \in [0,T]$, 
\begin{eqnarray*}
\|u^J(t_1) -u^J(t_2)\|^2_{L^2(X)} \leq \Big\| \sum_{j=0}^J \big( u_j(t_1)-u_j(t_2) \big) \phi_j \Big\|^2_{L^2(X)} =\sum_{j=0}^J |u_j(t_1)-u_j(t_2)|^2.
\end{eqnarray*}
Using the specific form of $u_j$ in \eqref{eq-uj} and $u_j(t)\in C^1 ([0,T])$, we see
\begin{eqnarray} \label{eq-time}
|u_j(t_1)-u_j(t_2)|^2
\leq C(T) \Big( \lambda_j \psi^2_{0,j} + \psi^2_{1,j} + \int_0^T |f_j(\tau)|^2 \di\tau \Big) |t_1-t_2|^2.
\end{eqnarray}
Then for $t_1\neq t_2$, we have
\begin{eqnarray} \label{eq-time-Lip}
\frac{\|u^J(t_1) -u^J(t_2)\|^2_{L^2(X)}}{|t_1-t_2|^2} &\leq& C(T) \sum_{j=0}^J \Big( \lambda_j \psi^2_{0,j} +\psi_{1,j}^2 + \int_0^T |f_j(\tau)|^2 \di\tau \Big) \nonumber \\
&\leq & C(T) \Big( \|\psi_0\|^2_{H^{1,2}(X)} +  \|\psi_1\|^2_{L^2(X)} + \|f\|^2_{L^2(X\times [0,T])} \Big).
\end{eqnarray}
This proves $u^J$ is in the energy class with desired energy estimates.

\smallskip
Next, we verify that $u^J$ is a weak (actually strong) solution of the wave equation.
Since $f\in C([0,T]; L^2(X))$ and $s_j(0)=0$, then $u_j(t)\in C^{2}([0,T])$. Using $s''_j(t)+\lambda_j s_j(t)=0$, we have
\begin{eqnarray*}
(\partial_t^2 u_j) (t) &=& \psi_{0,j} s'''_j(t) + \psi_{1,j} s''_j(t)+ f_j(t) + \int_0^t f_j(\tau) s''_j(t-\tau) \di\tau \\
&=& -\psi_{0,j} \lambda_j s'_j(t) -\psi_{1,j} \lambda_j s_j(t) + f_j(t) -\int_0^t f_j(\tau) \lambda_j s_j(t-\tau) \di\tau \\
&=& -\lambda_j u_j(t) +f_j(t).
\end{eqnarray*}
Hence $u^J$ satisfies the wave equation with source $f^J:=\sum_{j=0}^J f_j(t) \phi_j(x)$ strongly:
\begin{eqnarray*}
(\partial_t^2 u^J)(x,t) = \sum_{j=0}^J \big(\partial_t^2 u_j(t) \big) \phi_j(x) = \sum_{j=0}^J \big(f_j(t)-\lambda_j u_j(t) \big)\phi_j(x)=\Delta_X u^J +f^J.
\end{eqnarray*}
To check the initial conditions, 
$$u^J|_{t=0}=\sum_{j=0}^J \psi_{0,j} \phi_j(x),\quad \partial_t u^J|_{t=0}=\sum_{j=0}^J \psi_{1,j} \phi_j(x).$$
Denote the initial conditions above by $\psi_0^J,\psi_1^J$, respectively.
This shows $u^J$ is a strong solution of the wave equation for the source $f^J$ and initial conditions $\psi_0^J, \psi_1^J$.
To check the integral identity \eqref{def-weak}, we use the definition of $\Delta_X$ and integration by parts. Indeed, for any test function $v\in L^{\infty}([0,T]; H^{1,2}(X)) \cap {\rm Lip} ([0,T]; L^2(X))$ satisfying $v|_{t=T}=0$, we have
\begin{eqnarray*}
\int_0^T (\partial_t^2 u^J) v \,\di t = (\partial_t u^J) v|_{t=0}^T- \int_0^T (\partial_t u^J) (\partial_t v) \di t= - v|_{t=0} \psi_1^J - \int_0^T (\partial_t u^J) (\partial_t v) \di t.
\end{eqnarray*}
This is justified since $(\partial_t u^J) v \in {\rm Lip} ([0,T]; L^2(X))$ due to $u^J\in C^2([0,T];\textrm{Lip}(X))$.

\medskip
{\bf Step 2.} We show that $u^J\to u$ in $L^{\infty}([0,T]; H^{1,2}(X))$, and $(\partial_t u^J)(t) \to (\partial_t u)(t)$ in $L^2(X)$ uniformly for almost all $t$.

\smallskip
Since we have shown in \eqref{eq-H1N-estimate} that the $H^{1,2}$-norm of $u^J$ is uniformly bounded independently of $J$: namely for each $t$,
$$\|u^J(t)\|^2_{H^{1,2}(X)} =\sum_{j=0}^J (1+\lambda_j) u^2_j(t) \leq C(T) \Big( \|\psi_0\|^2_{H^{1,2}(X)} +  \|\psi_1\|^2_{L^2(X)} + \|f\|^2_{L^2(X\times [0,T])} \Big),$$
so let $J\to \infty$, using \eqref{def-H1-norm}, we have $u(t)\in H^{1,2}(X)$ for each $t$ with uniformly bounded norm, i.e., $u\in L^{\infty}([0,T]; H^{1,2}(X))$.
Moreover, as $\psi_{0}\in H^{1,2}(X)$ and $\psi_1\in L^2(X)$,
\begin{eqnarray} \label{eq-H1-uniform}
\sup_{0\leq t\leq T}\|u(t)-u^J(t)\|_{H^{1,2}(X)}&=&\sup_{0\leq t\leq T} \sum_{j=J+1}^{\infty} (1+\lambda_j)u_j(t)^2 \nonumber \\
&\leq & \sup_{0\leq t\leq T} C(T) \sum_{j=J+1}^{\infty} \Big( (1+\lambda_j) \psi_{0,j}^2 + \psi_{1,j}^2 +\int_{0}^t f^2_j(\tau) \di \tau \Big) \nonumber \\
&\leq & C(T) \sum_{j=J+1}^{\infty} \Big( (1+\lambda_j) \psi_{0,j}^2 + \psi_{1,j}^2 +\int_{0}^T f^2_j(\tau) \di\tau \Big) \to 0,
\end{eqnarray}
as $J\to \infty$ by the dominated convergence theorem.
This proves that $u^J\to u$ in $L^{\infty}([0,T]; H^{1,2}(X))$.

\smallskip

Next we consider the convergence of $\partial_t u^J$. Since $u_j(t)\in C^2([0,T])$, using \eqref{eq-time-Lip},
\begin{eqnarray*}
\|(\partial_t u^J)(t)\|^2_{L^2(X)} &=& \Big \| \sum_{j=0}^J u'_j(t) \phi_j(x) \Big\|^2_{L^2} =\sum_{j=0}^J |u'_j(t)|^2 \\
&\leq & C(T) \Big( \|\psi_0\|^2_{H^{1,2}} +  \|\psi_1\|^2_{L^2} + \|f\|^2_{L^2(X\times [0,T])} \Big).
\end{eqnarray*}
This shows
\begin{equation} \label{eq-time-conv}
(\partial_t u^J) (t) \to \sum_{j=0}^{\infty} u'_j(t) \phi_j \in L^2(X), \quad {\rm as}\; J\to \infty,
\end{equation}
in $L^2(X)$ for each $t$. Moreover, by the same consideration as in \eqref{eq-H1-uniform}, the convergence is uniform in $t$.

Since $u^J\in  {\rm Lip}([0,T];L^2(X))$ with uniform Lipschitz constant and $u^J\to u$ in $L^{\infty}([0,T];H^{1,2}(X))$, then $u\in {\rm Lip}([0,T];L^2(X))$.
It follows that $(\partial_t u)(t) \in L^2(X)$ exists for almost all $t$ (see e.g. \cite[Cor. 1.2.7]{ABHN}), namely
\begin{eqnarray}
\lim_{h\to 0}\Big\| \frac{u(t+h)-u(t)}{h} - (\partial_t u)(t) \Big\|_{L^2(X)} = 0.
\end{eqnarray}
This implies that, for almost all $t$,
$$\int_X (\partial_t u) \phi_j \di\meas = \lim_{h\to 0} \int_X \frac{u(t+h)-u(t)}{h} \phi_j \di\meas = \lim_{h\to 0} \frac{u_j(t+h)-u_j(t)}{h} = u'_j(t).$$
Then for almost all $t$, we have
\begin{eqnarray}
(\partial_t u)(t) = \sum_{j=0}^{\infty} \langle \partial_t u,\phi_j \rangle_{L^2} \phi_j(x) = \sum_{j=0}^{\infty} u'_j(t) \phi_j(x).
\end{eqnarray}
Combining with \eqref{eq-time-conv} gives $(\partial_t u^J)(t) \to (\partial_t u) (t)$ in $L^2(X)$ for almost all $t$ uniformly.

\medskip
{\bf Step 3.} Verify that $u$ is a weak solution of the wave equation.

\smallskip
From \textit{Step 1}, $u^J$ satisfies the integral identity
\begin{equation} \label{eq-weak-uN}
\int_0^T \int_X \Big(-\partial_t u^J\, \partial_t v+ \langle \nabla u^J,\nabla v\rangle \Big) \di\meas \di t= \int_0^T \int_X f^J v +\int_X v|_{t=0} \psi_1^J,
\end{equation}
for all test functions $v\in L^{\infty}([0,T]; H^{1,2}(X)) \cap {\rm Lip} ([0,T]; L^2(X))$ satisfying $v|_{t=T}=0$.

Since $u^J(t) \to u(t)$ in $H^{1,2}(X)$ uniformly in $t$ by \textit{Step 2}, 
we see that $\int_X \langle \nabla u^J(t) ,\nabla w \rangle \to \int_X \langle \nabla u(t) ,\nabla w \rangle$ for all $w\in H^{1,2}(X)$ uniformly in $t$. 
As $(\partial_t u^J)(t)\to (\partial_t u)(t)$ in $L^2(X)$ for almost all $t$ uniformly by \textit{Step 2}, and $f^{J} \to f$ in $L^2(X\times [0,T])$, $\psi_1^J \to \psi_1$ in $L^2(X)$, the integral identity \eqref{eq-weak-uN} converges to that for $u$. 
To check the initial values, $u|_{t=0}=\sum_{j=0}^{\infty} \psi_{0,j} \phi_j(x)=\psi_0.$
This concludes the proof for the existence of weak solution.

\medskip
{\bf Uniqueness.} To check the uniqueness of weak solution in the energy class, since $\langle \nabla u, \nabla v \rangle$ is a bilinear form in RCD spaces, it suffices to prove that the solution with $\psi_0=\psi_1=0$ and $f=0$ is $u=0$. 
For $\eta\in [0,T]$, we choose the test functions
\begin{equation} \label{eq-vj}
v_{j, \eta}(t) = \left\{ \begin{aligned}
&\phi_j(x) s_j(\eta-t) , \quad & t\in [0,\eta], \\
& 0, \quad & t\in (\eta,T],
\end{aligned} \right.
\end{equation}
where $s_j(t)$ is defined in \eqref{eq-sj}, and $\phi_j$ are the orthonormalized eigenfunctions of $\Delta_X$.
Note that $v_{j,\eta}\in {\rm Lip}([0,T];D(\Delta_X)) \subset L^{\infty}([0,T]; H^{1,2}(X)) \cap {\rm Lip} ([0,T]; L^2(X))$ and satisfies $v_{j,\eta}|_{t=T}=0$.
Using the weak formulation with test functions $v_{j,\eta}$,
$$\int_0^{\eta} \int_X (-\partial_t u)(\partial_t v_{j,\eta}) \di t \di\meas + \int_0^{\eta} \int_X \langle \nabla u, \nabla v_{j,\eta} \rangle \di t \di\meas=0.$$
Then
\begin{eqnarray*}
-\int_X u(\partial_t v_{j,\eta})\big|_{0}^{\eta} + \int_0^{\eta} \int_X u(\partial^2_t v_{j,\eta})- \int_0^{\eta} \int_X u \Delta_X v_{j,\eta}=0.
\end{eqnarray*}
Since $u|_{t=0}=0$, $(\partial_t v_{j,\eta})|_{t=\eta}=-\phi_j$, $\partial^2_t v_{j,\eta}=-\lambda_j v_{j,\eta}=\Delta_X v_{j,\eta}$, we have $\int_X u(x,\eta) \phi_j(x) \di \meas(x)=0$ for all $j\in \mathbb{N}$. Hence $u(\cdot,\eta)=0$ for all $\eta\in [0,T]$.
\end{proof}



\begin{corollary} \label{coro-wave-extension}
Let $(X,\dist,\meas)$ be a compact {\rm RCD}$(K,N)$ space for some $K \in \mathbb{R}$ and $N \in [1, \infty)$.
Let $T>0$, $f \in C([0,2T]; L^2(X))$, $\psi_0\in  H^{1,2}(X)$, $\psi_1\in L^2(X)$.
Let $u\in E([0,T];X)$ be the weak solution of the wave equation
$$(\partial_t^2-\Delta_X)u=f,\quad u|_{t=0}=\psi_0,\quad \partial_t u|_{t=0}=\psi_1.$$
Suppose that $\partial_t u$ exists at $t=T$ in the strong sense.
Consider the wave equation $(\partial_t^2-\Delta_X)w=f$ on $X\times [T,2T]$ and let $w\in E([T,2T];X)$ be the weak solution with initial conditions $u|_{t=T}$ and $\partial_t u|_{t=T}$ at $t=T$. Then
\begin{equation*} 
u_e=\left\{ \begin{aligned}
&u, \quad t\in [0,T], \\
& w, \quad t\in (T,2T],
\end{aligned} \right.
\end{equation*}
is the weak solution of the wave equation on $X\times [0,2T]$ with initial conditions $\psi_0$ and $\psi_1$ at $t=0$.
\end{corollary}

\begin{proof}
This can be argued using the approximate solutions $u^J$ constructed in the proof of Proposition \ref{prop-regularity}. Namely, let $u^J\in C^2([0,T];D(\Delta_X))$ be the approximate solutions for $u$, which satisfy
$$\partial_t^2 u^J-\Delta_X u^J =f^J.$$ 
Taking approximate solutions $w^J$ for $w$ and combining with $u^J$, we see that $u_e^J$ satisfies 
$$\partial_t^2 u_e^J-\Delta_X u_e^J =f^J, \quad \textrm{ for }t\in [0,T)\cup (T,2T].$$
Since the initial conditions for $w$ match $u|_{t=T}, \partial_t u|_{t=T}$, and $(\partial_t u)^J = \partial_t u^J$ if $\partial_t u$ exists, 
then $u_e^J \in C^1([0,2T]; D(\Delta_X))$ so $\partial_t u_e^J \in C([0,2T];D(\Delta_X))$.
Hence, for any test function $v\in E([0,2T];X)$ satisfying $v|_{t=2T}=0$, by splitting the time domain into $(0,T)$ and $(T,2T)$, we see that the corresponding terms at $t=T$ cancel out:
\begin{eqnarray*}
\int_0^{2T} \int_X \langle \nabla u_e^J ,\nabla v \rangle -\int_0^{2T} \int_X f^J v = -\int_0^{2T} \int_X (\partial_t^2 u_e^J)v = \int_0^{2T} \int_X (\partial_t u_e^J)(\partial_t v) +\int_X v|_{t=0} (\partial_t u^J)|_{t=0}.
\end{eqnarray*}
Let $J\to \infty$ and the claim follows.
\end{proof}


\begin{theorem}[Finite speed of propagation]
\label{prop-finite-propagation}
Let $(X,\dist,\meas)$ be a compact {\rm RCD}$(K,N)$ space for some $K \in \mathbb{R}$ and $N \in [1, \infty)$. Consider the wave equation \eqref{eq-wave-intro} with initial conditions $\psi_0\in H^{1,2}(X)$, $\psi_1\in L^2(X)$ and source $f\in C([0,T]; L^2(X))$.
Suppose that for some $x_0\in X$ and $r>0$, $\psi_0|_{B(x_0,r)}=\psi_1 |_{B(x_0,r)}=0$, and $f=0$ on $K(x_0,r):=\{(x,t)\in X\times (0,r): \dist (x,x_0)<r-t\}$.
Then the weak solution $u$ satisfies $u=0$ on $K(x_0,r)$.
\end{theorem}

\begin{proof}
Let $u^J\in C^2([0,T];D(\Delta_X)\cap \textrm{Lip}(X))$ be the approximate solution and define 
$$E^J(t):= \int_{B(x_0,r-t)} \Big(|u_t^J(t)|^2 + |\nabla u^J(t)|^2 \Big) \di\meas, \quad\textrm{ for }t\in [0,r).$$
Define $E(t)$ by the same formula with $u^J$ replaced by $u$.
Then $E^J(t)\to E(t)$ as $J\to \infty$ for almost all $t$.
By the initial conditions, $E(0)=0$ by the locality of $|\nabla u|$.
We aim to prove that $E(t)\leq E(0)=0$ for almost all $t$, hence the conclusion. To this end, we verify $E^{J}(t)$ is absolutely continuous and estimate its derivative.

Let $0\leq t_1<t_2< r$ be arbitrary. Consider
\begin{eqnarray*}
E^J(t_2)-E^J(t_1) 
&=& \int_{B(x_0,r-t_2)} \Big(|u_t^J(t_2)|^2 + |\nabla u^J(t_2)|^2 \Big)  - \int_{B(x_0,r-t_2)} \Big(|u_t^J(t_1)|^2 + |\nabla u^J(t_1)|^2 \Big) \\
&&+ \int_{B(x_0,r-t_2)} \Big(|u_t^J(t_1)|^2 + |\nabla u^J(t_1)|^2 \Big)  - \int_{B(x_0,r-t_1)} \Big(|u_t^J(t_1)|^2 + |\nabla u^J(t_1)|^2 \Big) \\
&=& A_1(t_2,t_1)+A_2(t_2,t_1).
\end{eqnarray*}
By the coarea formula, e.g. \cite[Cor. 1.9]{ABS}, applied to distance function $\dist (\cdot,x_0)$, the ball $B(x_0,r-t)$ has finite perimeter for almost all $t$, and
\begin{eqnarray*}
A_2(t_2,t_1)=-\int_{t_1}^{t_2} \int_{S(x_0,r-t)} \Big(|u_t^J(t_1)|^2 + |\nabla u^J(t_1)|^2 \Big) \di|D\chi_{B(x_0,r-t)}| \di t,
\end{eqnarray*}
where $S(x_0,r-t)=\{x\in X: \dist (x,x_0)=r-t\}$.
Denote by 
\begin{equation}\label{def-G2}
G_2(t;t_1):=\int_{S(x_0,r-t)} \Big(|u_t^J(t_1)|^2 + |\nabla u^J(t_1)|^2 \Big) \di|D\chi_{B(x_0,r-t)}|, \textrm{ for any }t\in [0,r).
\end{equation}
Since $G_2(t;t_1)\in \mathcal{L}^1$ (in $t$) by the coarea formula for any $t_1$, then setting $t_1=0$, $A_2(\cdot, 0)$ is an absolutely continuous function.

On the other hand, 
\begin{eqnarray} \label{eq-A2}
A_1(t_2,t_1) &=& \int_{B(x_0,r-t_2)} \int_{t_1}^{t_2} \frac{\partial}{\partial t} \Big(|u_t^J(t)|^2 + |\nabla u^J(t)|^2 \Big) \di t\, \di\meas \nonumber \\
&=& \int_{t_1}^{t_2} \int_{B(x_0,r-t_2)} \Big(2 u_t^J(t) u_{tt}^J (t) + 2\langle \nabla u_t^J(t) , \nabla u^J(t) \rangle \Big) \di\meas\,\di t = \int_{t_1}^{t_2}  G_1(t;t_2) \di t,
\end{eqnarray}
where 
\begin{equation}
G_1(t;t_2):= \int_{B(x_0,r-t_2)} \Big(2 u_t^J(t) u_{tt}^J (t) + 2\langle \nabla u_t^J(t) , \nabla u^J(t) \rangle \Big) \di \meas.
\end{equation}
Again setting $t_1=0$, we see that
\begin{eqnarray*}
A_1(t_2,0) &=& \int_0^{t_2} G_1(t;0) \di t + \int_0^{t_2} \big( G_1(t;t_2) -G_1(t;0) \big) \di t \\
&=& \int_0^{t_2} G_1(t;0) \di t - \int_0^{t_2} \Big(\int_0^{t_2} \int_{S(x_0,r-s)} \big( 2 u_t^J(t) u_{tt}^J (t) + 2\langle \nabla u_t^J(t) , \nabla u^J(t) \rangle \big) \di |D\chi|  \di s \Big)    \di t.
\end{eqnarray*}
Since $u^J\in C^2([0,T];\textrm{Lip}(X))$, all functions appearing above are bounded a.e., 
and hence $A_1(t_2,0)$ can be written as the integral of an integrable function (after switching the order of integration in the latter term), i.e., $A_1(\cdot,0)$ is absolutely continuous.
Therefore, $E^J(\cdot)$
is absolutely continuous so differentiable almost everywhere.

We consider
\begin{equation}
H(t):=\int_{0}^t G_2(\tau;\tau) \di \tau.
\end{equation}
Since $|\nabla u^J|$ is bounded a.e. with respect to the 2-capacity by \cite[Cor. 2.9]{BPS}, 
so a.e. with respect to the perimeter measure by \cite[Lem. 1.10 and Th. 1.12]{BPS03}, then $G_2(\tau;\tau)$ is integrable. 
Hence by the Lebesgue differentiation theorem, $H(t)$ is differentiable for a.e. time $t$ with $(\frac{\di }{\di t} H)(t)=G_2(t;t)$.

\smallskip
From now on, we take a.e. $t_1\in [0,r)$ to be any time at which $E^J(t)$ is differentiable and simultaneously at which $(\frac{d}{dt} H)(t_1)=G_2(t_1;t_1)$.
Then
\begin{eqnarray*}
\big(\frac{d}{dt} E^J \big)(t_1)&=& \lim_{t_2\to t_1} \frac{E^J(t_2)-E^J(t_1)}{t_2-t_1} \\
&=&\lim_{t_2\to t_1} \frac{1}{t_2-t_1} \int_{t_1}^{t_2} \big( G_1(t;t_2) - G_2 (t;t_1) \big) \di t.
\end{eqnarray*}
For the first term, as $G_1(t;t_2)$ is continuous in both variables, we have
$$\lim_{t_2\to t_1} \frac{1}{t_2-t_1} \int_{t_1}^{t_2} G_1(t;t_2) \di t = G_1(t_1;t_1).$$
For the second term, since $u_t^J(t_1),\nabla u^J(t_1)$ in $G_2(t;t_1)$ are Lipschitz in $t_1$ and $|\nabla u^J|$ is bounded a.e. with respect to the perimeter measure, using the choice of $t_1$, we have
\begin{eqnarray*}
\lim_{t_2\to t_1} \frac{1}{t_2-t_1} \int_{t_1}^{t_2} G_2(t;t_1) \di t &=& \lim_{t_2\to t_1} \frac{1}{t_2-t_1} \int_{t_1}^{t_2} G_2(t;t) \di t + \lim_{t_2\to t_1} \frac{1}{t_2-t_1} \int_{t_1}^{t_2} \big(G_2(t;t_1)-G_2(t;t) \big) \di t \\
&=& (\frac{d}{dt} H)(t_1) = G_2(t_1;t_1),
\end{eqnarray*}
where we used
\begin{eqnarray*}
\lim_{t_2\to t_1} \frac{1}{t_2-t_1} \int_{t_1}^{t_2} \big| G_2(t;t_1)-G_2(t;t) \big| \di t &\leq& \lim_{t_2\to t_1} \frac{1}{t_2-t_1} \int_{t_1}^{t_2} \int_{S(x_0,r-t)} C|t-t_1| \di |D\chi_{B(x_0,r-t)}| \di t \\
&\leq& \lim_{t_2\to t_1} \frac{1}{t_2-t_1} \int_{t_1}^{t_2} \int_{S(x_0,r-t)} C|t_2-t_1| \di |D\chi_{B(x_0,r-t)}| \di t =0.
\end{eqnarray*}
Thus, we conclude that
\begin{equation} \label{eq-diff-EJ}
\big(\frac{d}{dt} E^J \big)(t_1)=G_1(t_1;t_1) - G_2(t_1;t_1), \quad \textrm{for a.e. }t_1\in [0,r).
\end{equation}

\medskip


Next, we estimate $G_1(t_1;t_1)$ as follows.
\begin{eqnarray*}
G_1(t_1;t_1)
= \int_{B(x_0,r-t_1)}  \Big(2 u_t^J(t_1) u_{tt}^J (t_1) + 2\langle \nabla u_t^J(t_1) , \nabla u^J(t_1) \rangle \Big) \di \meas.
\end{eqnarray*}
For each fixed $t$, both $u_t^J$ and $u^J$ are test functions since they are linear combinations of eigenfunctions.
Recall the divergence formula
$$\textrm{div} (u_t^J \nabla u^J)= \langle \nabla u_t^J , \nabla u^J \rangle + u_t^J \Delta_X u^J,$$
and the Gauss-Green formula \cite[Th. 2.4]{BPS03}, 
$$\int_{B(x_0,r-t_1)} \textrm{div} (u_t^J \nabla u^J) \di\meas= -\int \langle \textrm{tr}(u_t^J \nabla u^J) , \nu \rangle \,\di|D\chi_{B(x_0,r-t_1)}|,$$
for a.e. $t_1$, and $\nu\in L^2_{B(x_0,r-t_1)}(TX)$ is the unit normal satisfying $|\nu|=1$ a.e. with respect to the perimeter measure.
Then using $u_{tt}^J=\Delta_X u^J+f^J$ and
\begin{equation} \label{eq-trace}
|\langle \textrm{tr}(u_t^J \nabla u^J), \nu \rangle| \leq |\textrm{tr}(u_t^J \nabla u^J)| \cdot |\nu| =\textrm{tr}(|u_t^J \nabla u^J|)=|u_t^J| \cdot |\nabla u^J| 
\end{equation}
which satisfies a.e. with respect to the perimeter measure, we have
\begin{eqnarray} \label{dif-A1}
G_1(t_1;t_1)&=&2 \Big( \int_{B(x_0,r-t_1)}u_t^J u_{tt}^J - \int \langle \textrm{tr}(u_t^J \nabla u^J) , \nu \rangle \,\di|D\chi_{B(x_0,r-t_1)}|- \int_{B(x_0,r-t_1)} u_t^J \Delta_X u^J \Big) \nonumber\\
&\leq& 2 \Big( \int_{B(x_0,r-t_1)}u_t^J f^J + \int  |u_t^J| \cdot |\nabla u^J| \,\di|D\chi_{B(x_0,r-t_1)}|\Big) \nonumber\\
&\leq &  \int_{S(x_0,r-t_1)} \big(|u^J_t(t_1)|^2+|\nabla u^J(t_1)|^2 \big)\,\di|D\chi_{B(x_0,r-t_1)}| + 2 \int_{B(x_0,r-t_1)}u_t^J(t_1) f^J(t_1).
\end{eqnarray}
The inequality \eqref{eq-trace} is due to the following facts. The first inequality uses the definition of $L^2$-normed $L^{\infty}$-module on the perimeter, which is a Hilbert module so the Cauchy–Schwarz inequality is satisfied \cite{Gigli_lecture_notes}.
The second equality uses the property that $|\textrm{tr}(w)|=\textrm{tr}(|w|)$ a.e. for test vector field $w=u_t^J \nabla u^J$, see e.g. \cite[Sec. 2]{BPS03}. 
The third equality follows from the definition of trace operator and the quasi-continuous representative given by \cite[Cor. 2.9]{BPS}.
In particular $|\nabla u^J|$ is bounded a.e. with respect to the perimeter measure by \cite[Lem. 1.10 and Th. 1.12]{BPS03}.

Hence by \eqref{def-G2}, \eqref{dif-A1} and \eqref{eq-diff-EJ}, for a.e. $t_1\in [0,r)$, we have
\begin{equation} \label{diff-EN}
\big(\frac{d}{dt} E^J \big) (t_1)= G_1(t_1;t_1)-G_2(t_1;t_1)
\leq 2\int_{B(x_0,r-t_1)} u^J_t (t_1) f^J (t_1).
\end{equation}
Since $E^J(t)$ is absolutely continuous, using the estimate \eqref{diff-EN}, 
\begin{eqnarray*}
E^J(t)&=& E^J(0)+\int_0^t \frac{d}{d\tau} E^J(\tau) \di\tau \\
&\leq& E^J(0) + 2\int_0^t \int_{B(x_0,r-\tau)} u^J_t(\tau) f^J(\tau) \di\meas \di t. 
\end{eqnarray*}
Since $u_t^J \to u_t$ and $f^J\to f$ uniformly for almost all time, and $E^J(0)\to E(0)$, then using the condition $f|_{K(x_0,r)}=0$, we have
\begin{equation}
E(t)\leq E(0)+2\int_0^t \int_{B(x_0,r-\tau)} u_t(\tau) f(\tau) \di\meas \di\tau=E(0)=0,
\end{equation}
for almost all $t\in [0,r)$. 
As $u\in \textrm{Lip}([0,T];L^2(X))$, this shows $u=0$ in $K(x_0,r)$.
\end{proof}

\section{Unique continuation on RCD spaces} \label{sec-uc}

We state the unique continuation for wave operator on a class of almost Lipschitz spaces.

\begin{theorem}
\label{theorem-uc-general}
Let $(X,\dist,\meas)$ be a compact {\rm RCD}$(K,N)$ space for some $K \in \mathbb{R}$ and some $N \in [1, \infty)$, and let $\mathcal{R}\subset X$ be an open set whose complement is $\meas$-null.
Assume that
\begin{itemize}
\item[(1)] $\mathcal{R}$ is an open weighted Lipschitz Riemannian manifold with locally Lipschitz density $\rho$; 
\item[(2)] For any $x,y\in \mathcal{R}$, there exists a shortest path contained in $\mathcal{R}$ between $x,y$.
\end{itemize}
Then the Tataru-type optimal (weak) unique continuation for the wave operator holds on $(X,\dist,\meas)$. Namely, if $u\in L^{\infty}([-T,T]; H^{1,2}(X)) \cap {\rm Lip} ([-T,T]; L^2(X))$ is a weak solution of $(\partial_t^2-\Delta_X)u=0$ and $u|_{V\times (-T,T)}=0$ on an open subset $V\subset \mathcal{R}$, then $u=0$ on the double cone of influence
\begin{equation}
K(V,T):=\big\{(x,t)\in \mathcal{R}\times (-T,T): \dist(x,V) <T-|t| \big\}.
\end{equation}
\end{theorem}
\begin{corollary}
\label{theorem-uc}
The unique continuation for wave operator holds on non-collapsed Ricci limit spaces with two-sided Ricci curvature bounds, on cross sections of tangent cones of such spaces, and on Einstein orbifolds.
\end{corollary}

Note that the convexity assumption (2) in Theorem \ref{theorem-uc-general} can be relaxed to weakly convex by Corollary \ref{theorem-uc-weak}.
In this section, we prove the unique continuation result above by propagating local unique continuation along shortest paths between regular points of $X$. 
Using the classical theory of Carleman-Tataru estimates \cite{T95,EINT}, for the unique continuation for the wave operator to hold, the metric (or coefficient in the principle term) is typically required to be $C^1$, which can be relaxed to Lipschitz, see Appendix \ref{sec-Lip}.
For non-collapsed Ricci limit spaces with two-sided curvature bounds, the regular set is a $C^{1,\alpha}$-Riemannian manifold by \cite[Th. 7.2]{CC1} and is convex by \cite[Th. 1.5]{CN12}.

First, we state the simplest case where the unique continuation does not reach the singular set. In such case, the unique continuation reduces to the classical case for Riemannian manifolds.

\begin{lemma} 
\label{uc-Tataru}
Let $(X,\dist,\meas)$ be a compact $\RCD(K, N)$ space satisfying the assumptions of Theorem \ref{theorem-uc-general}.
Assume $\overline{V} \subset  \mathcal{R}$ and $T<\dist(V,\mathcal{S})$, where $\mathcal{S}:=X\setminus \mathcal{R}$.
If $u\in L^{\infty}([-T,T]; H^{1,2}(X)) \cap {\rm Lip} ([-T,T]; L^2(X))$ is a weak solution of $(\partial_t^2-\Delta_X)u=0$ and $u|_{V\times (-T,T)}=0$, then $u=0$ on $K(V,T)$.
\end{lemma}

\begin{proof}
Under the assumptions of Theorem \ref{theorem-uc-general}, let $g$ be the Lipschitz Riemannian metric on $\mathcal{R}$, and $\rho\in \textrm{Lip}(\mathcal{R})$ be the density function such that $\di\meas=\rho\, \di {\rm vol}_{g}$.
Given any regular point, there is a local coordinate in a neighborhood such that the components $g_{ij}(x)$ is Lipschitz in this coordinate.
Since $T<\dist(V,\mathcal{S})$ so unique continuation does not propagate to the singular set. It suffices to reduce this case to the classical Riemannian case.

Let $\chi\in C^{1,1}(\mathcal{R})$ be a cut-off function compactly supported in $\mathcal{R}$. 
Then using test function $\chi v$ in the weak formulation (Def. \ref{weak-formulation}) and the Leibniz rule,
\begin{eqnarray*}
\int_X (\chi v) |_{t=0} \psi_1 &=& \int_0^T \int_X \Big(-\partial_t u\, \partial_t (\chi v)+ \langle \nabla u,\nabla (\chi v)\rangle \Big) \di\meas \di t \\
&=& -\int_0^T \int_X \partial_t (\chi u)\, \partial_t v + \int_0^T \int_X \Big( \langle \nabla u,\nabla \chi \rangle v + \langle \nabla u,\nabla v \rangle \chi \Big) \\
&=& -\int_0^T \int_X \partial_t (\chi u)\, \partial_t v + \int_0^T \int_X \Big( \langle \nabla (\chi u),\nabla v \rangle - \langle \nabla \chi,\nabla v \rangle u + \langle \nabla u,\nabla \chi \rangle v \Big).
\end{eqnarray*}
Now if we restrict the test functions $v\in {\rm Lip}(X\times [0,T])$, then $uv\in L^{\infty} ([0,T];H^{1,2}(X))$. Then by the definition of $\Delta_X$,
\begin{eqnarray*}
\int_X v |_{t=0} (\chi \psi_1) &=& -\int_0^T \int_X \partial_t (\chi u)\, \partial_t v + \int_0^T \int_X \Big( \langle \nabla (\chi u),\nabla v \rangle - \langle \nabla \chi,\nabla (uv) \rangle  + 2\langle \nabla u,\nabla \chi \rangle v \Big) \\
&=& -\int_0^T \int_X \partial_t (\chi u)\, \partial_t v + \int_0^T \int_X \Big( \langle \nabla (\chi u),\nabla v \rangle + (\Delta_X \chi )uv + 2\langle \nabla u,\nabla \chi \rangle v \Big).
\end{eqnarray*}
Let us assume that $U={\rm supp}(\chi)\supset \overline{V}$ satisfying $\overline{U}\subset \mathcal{R}$.
We observe that the above integral identity is identically zero outside of $U$ by the locality of $|\nabla u|$.
Thus, it suffices to take the test functions $v$ such that ${\rm supp}(v)\subset U\subset \mathcal{R}$.
This shows that the localized function 
\begin{equation}
\widetilde{u}:=\chi u\in L^{\infty}([0,T]; H_0^{1,2}(U)) \cap {\rm Lip} ([0,T]; L^2(U))
\end{equation}
is a weak solution of the hyperbolic equation 
\begin{equation} \label{eq-local-u}
(\partial_t^2 - \Delta_X)\widetilde{u} = -2\langle \nabla \chi, \nabla u \rangle-u\Delta_X\chi,
\end{equation}
with test functions $v\in {\rm Lip}(U\times [0,T])$ satisfying $v|_{t=T}=0$.

\smallskip
In the subset $U\subset \mathcal{R}$, the space $H^{1,2}(U)$ defined in the non-smooth setting coincides with the classical Sobolev space defined by the Riemannian structure of $\mathcal{R}$, see e.g. \cite[Sec. 3.1]{H18}. Moreover, since Lipschitz functions are dense in $H^{1,2}(X)$ by \cite[Prop. 4.10]{AGS},
and for Lipschitz functions the Cheeger energy coincides with the $L^2$-norm of local Lipschitz constant \cite{Ch99}, 
the Laplacian $\Delta_X$ defined in the non-smooth setting coincides with the weighted Laplacian $\Delta_{g,\rho}$ on the weighted Riemannian manifold $\mathcal{R}$ with locally Lipschitz density $\rho$ (see for instance the proof of \cite[Th. 4.1]{H20} with \cite[Th. 4.1]{BGHZ}). 
In the local coordinates on $\mathcal{R}$,
\begin{eqnarray}\label{Laplacian-local}
\Delta_{g,\rho} u&=& \frac{1}{\rho |g|^{\frac12}} \sum \frac{\partial}{\partial x^i} \big(\rho |g|^{\frac12} g^{ij}\frac{\partial u}{\partial x^j}\big), \quad |g|=\det(g_{ij}) \\ 
&=&\sum g^{ij}\frac{\partial^2 u}{\partial x^i \partial x^j} + \textrm{lower order terms}. \nonumber
\end{eqnarray}
Note that $\rho$ is nonzero at any point in $\mathcal{R}$ because one can apply the same proof of \cite[Th. 4.6]{CC3}.

As $g^{ij}$ and density function $\rho$ are Lipschitz so the coefficients in the lower order terms are in $L^{\infty}(U)$,
we apply the Carleman-Tataru estimate, Proposition \ref{prop-Carleman-Lip}, to the localized function $\widetilde{u}\in H^1(U\times [0,T])$ and the hyperbolic operator $P=\partial_t^2-\Delta_{g,\rho}$, where $\Delta_{g,\rho}$ is the local form of $\Delta_X$ on $U$, to prove the local unique continuation across any non-characteristic hypersurface $\Gamma:=\{y\in U\times (0,T):\psi(y)=0\}$ at $y_0\in \Gamma$.
Namely, by \eqref{eq-Carleman-Lip},
for sufficiently large $\tau$ and sufficiently small $\epsilon$,
\begin{eqnarray}
\label{carleman}
 \|e^{-\epsilon D_0^2/2\tau} e^{\tau \phi} \widetilde{u}\|_{1,\tau} \le C  \tau^{-1/2}\|e^{-\epsilon D_0^2/2\tau} e^{\tau \phi} P\widetilde{u}\|_{L^2} + C \tau^{-1/2} e^{-d\tau/2} \|e^{\tau \phi} \widetilde{u}\|_{1,\tau},
\end{eqnarray}
where $e^{-\epsilon D_0^2/2\tau}$ is the integral operator in the time variable with the kernel $(\tau/2\pi\epsilon)^{1/2} e^{-\tau|t'-t|/2\epsilon}$, and $\|u\|_{1,\tau}^2=\tau^2\|u\|_{L^2}^2+\|\nabla u\|_{L^2}^2$.
The quadratic weight function $\phi$ (in the chosen coordinate chart) is a pseudoconvex function, constructed according to the non-characteristic hypersurface $\Gamma$ such that
\begin{equation}\label{phi-support-condition}
\{y\in U\times (0,T):\psi(y) \leq 0\}\subset \{y\in U\times (0,T):\phi(y)<0\}\cup \{y_0\} \quad \textrm{ in }B(y_0,r),
\end{equation}
for sufficiently small $r>0$.
Without loss of generality, assume that $\chi=1$ in $B(y_0,r)$.

For the local unique continuation, suppose that $u$ vanishes on $\{\psi>0\}$, so does $\widetilde{u}$. Then $\textrm{supp}(\widetilde{u})\subset \{\phi<0\}\cup \{y_0\}$ by \eqref{phi-support-condition}, so the last term goes to zero as $\tau\to \infty$. 
By \eqref{eq-local-u}, $P\widetilde{u}$ vanishes on $B(y_0,r)$, so there exists some small $c>0$ such that $\textrm{supp}(P\widetilde{u})\subset \{\phi<-c\}$.
Then using the Carleman estimate \eqref{carleman}, for sufficiently large $\tau$,
$$\|e^{-\epsilon D_0^2/2\tau} e^{\tau \phi} \widetilde{u}\|_{1,\tau} \leq C\tau^{-1/2} e^{-c\tau}\|P\widetilde{u}\|_{L^2}\leq C\tau^{-1/2} e^{-c\tau}\|\widetilde{u}\|_{H^1},$$
which yields that $\|e^{-\epsilon D_0^2/2\tau} e^{\tau (\phi+c)} \widetilde{u}\|_{L^2}$ is bounded independent of $\tau$.
Letting $\tau\to \infty$, we have $\widetilde{u}=0$ on $\{\phi>-c\}$ which contains a neighborhood of $y_0$. This shows $u=0$ in a neighborhood of $y_0$ since $\chi=1$ in $B(y_0,r)$.

With the local unique continuation across any non-characteristic hypersurface, the optimal unique continuation follows by standard arguments (e.g. \cite[Sec. 3.4]{KKL}), as long as it does not reach the singular set. This is guaranteed by the assumption $T<\dist(V,\mathcal{S})$.
\end{proof}

A variation of Lemma \ref{uc-Tataru} gives the following.

\begin{lemma}
\label{uc-1}
Let $(X,\dist,\meas)$ be a compact $\RCD(K, N)$ space satisfying the assumptions of Theorem \ref{theorem-uc-general}.
Assume $\overline{V} \subset \mathcal{R}$ and $T>0$. 
If $u$ is a weak solution of $(\partial_t^2-\Delta_X)u=0$ and $u|_{V\times (-T,T)}=0$, then $u=0$ on $K(V,T)\cap \{x: \dist(x,V)<\dist(\mathcal{S},V)\}$.
\end{lemma}

\begin{proof}
Take an arbitrary point $(x_0,t_0)\in K(V,T)\cap \{x: \dist(x,V)<\dist(\mathcal{S},V)\}$ and we need to show $u(x_0,t_0)=0$.
Let $\epsilon>0$ such that $|t_0|+\dist(x_0,V)<T-\epsilon$ and $\dist(x_0,V)+\epsilon<\dist(\mathcal{S},V)$.
From the condition $|t_0|+\dist(x_0,V)<T-\epsilon$, we see that
$$u|_{V\times (t_0-\dist(x_0,V)-\epsilon,\, t_0+\dist(x_0,V)+\epsilon)}=0.$$
Let us translate the wave equation in time, defining $\hat{u}(x,t):=u(x,t+t_0)$. The wave equation is satisfied for $\hat{u}$ and $\hat{u}|_{V\times (-\dist(x_0,V)-\epsilon,\dist(x_0,V)+\epsilon)}=0$, with $\dist(x_0,V)+\epsilon<\dist(\mathcal{S},V)$. 
Note that we used the fact that for any $-T\leq t_1 <t_2 \leq T$, $u|_{[t_1,t_2]}$ is a weak solution of the wave equation for almost all $t_1$.
Applying Lemma \ref{uc-Tataru} gives $\hat{u}(x,t)=0$ for $(x,t)\in K(V, \dist(x_0,V)+\epsilon)$. This shows $u(x_0,t_0)=\hat{u}(x_0,0)=0$.
\end{proof}

Next, we prove Theorem \ref{theorem-uc-general} by propagating the unique continuation along a shortest path in $\mathcal{R}$.


\begin{proof}[Proof of Theorem \ref{theorem-uc-general}]
Without loss of generality, we can assume that $\overline{V} \subset \mathcal{R}$.
Take an arbitrary $(x_0,t_0)\in K(V,T)$ with $x_0\in \mathcal{R}$.
Let $\epsilon>0$ such that $|t_0|+\dist(x_0,V)<T-\epsilon$.
Pick a point $q\in V$ such that $\dist(x_0,q)<\dist(x_0,V)+\epsilon/2$.
As $q\in V\subset \mathcal{R}$ and $x_0\in \mathcal{R}$, consider a shortest path $[q x_0]$ that is contained in $\mathcal{R}$ due to Assumption (2) of Theorem \ref{theorem-uc-general}.
Let 
$$\sigma= \min\{ \dist([q x_0],\mathcal{S}) , \dist(V,\mathcal{S}) \}>0.$$
Pick points $\{q_j\}_{j=0}^J$ on the shortest path $[q x_0]$ such that
$$q_0=q,\quad q_J=x_0, \quad \dist(q_j,q_{j+1})<\sigma/4.$$

For point $q$, pick a ball $B(q,\delta)$ of small radius $\delta<\sigma/8$ such that $B(q,\delta)\subset V$.
Then the condition gives $u|_{B(q,\delta)\times (-T,T)}=0$.
Applying Lemma \ref{uc-1} gives 
$$u=0 \;\textrm{ in }\{(x,t): |t|<T-\dist(x,B(q,\delta)),\; \dist(x,B(q,\delta))< \sigma/2\}.$$
Notice that the right-hand set contains $B(q_1,\delta)\times \{t: |t|<T-\dist(q,q_1)-2\delta \}$.
Denote $T_1=T-\dist(q,q_1)-2\delta$.
Then applying Lemma \ref{uc-1} again gives
$$u=0 \;\textrm{ in }\{(x,t): |t|<T_1-\dist(x,B(q_1,\delta)),\; \dist(x,B(q_1,\delta))< \sigma/2\},$$
which contains
$$B(q_2,\delta)\times \{t: |t|<T_1-\dist(q_1,q_2)-2\delta=T-\dist(q,q_2)-4\delta \}.$$
Iterating this process, we obtain
$$u=0 \;\textrm{ in } B(q_j,\delta)\times \{t: |t|<T-\dist(q,q_j)-2j\delta \},$$
for all $j\in \{1,\cdots,J\}$ satisfying $T-\dist(q,q_j)-2j\delta>0$.

Now we choose $\delta<\epsilon/8J$ so that
$$T-\dist(q,x_0)-2J\delta>T-\dist(x_0,V)-\epsilon/2-\epsilon/4>T-\dist(x_0,V)-\epsilon>|t_0|\geq 0.$$
Thus, the process iterates the last step $j=J$, and in particular, $u(x_0,t_0)=0$.
\end{proof}

\begin{corollary} \label{theorem-uc-weak}
The Assumption {\rm (2)} in Theorem \ref{theorem-uc-general} can be relaxed to 
\begin{itemize}
\item[(2')] The open set $\mathcal{R}$ is weakly convex: namely, for any $x,y\in \mathcal{R}$ and any $\epsilon>0$, there exists an $\epsilon$-geodesic $\gamma\subset \mathcal{R}$ connecting $x,y$.
\end{itemize}
\end{corollary}

\begin{proof}
The proof is almost identical to that of Theorem \ref{theorem-uc-general}. 
Namely, for any $(x_0,t_0)\in K(V,T)$, there is $\epsilon>0$ such that $\dist(x_0,V)<T-|t_0|-4\epsilon$. By the relaxed assumption (2'), there is $q\in V$ and an $\epsilon$-geodesic $\gamma\subset \mathcal{R}$ connecting $q,x_0$ such that $L(\gamma)<\dist(x_0,V)+2\epsilon$. Then one can replace the geodesic $[qx_0]$ by the $\epsilon$-geodesic $\gamma$ in the proof of Theorem \ref{theorem-uc-general}, and pick points $q_0,\cdots,q_J$ on $\gamma$ in the same way with $q_0=q$, $q_J=x_0$.
Repeating the proof gives
$$u=0 \;\textrm{ in } B(q_j,\delta)\times \Big\{t: |t|<T-\sum_{k=1}^j \dist(q_{k-1},q_k)-2j\delta \Big\},$$
for all $j\in \{1,\cdots,J\}$ satisfying $T-\sum\limits_{k=1}^j \dist(q_{k-1},q_k)-2j\delta>0$.
To see the above iterates to the last step $j=J$, choosing $\delta<\epsilon/2J$,
\begin{eqnarray*}
T-\sum_{k=1}^J \dist(q_{k-1},q_k)-2J\delta > T-L(\gamma) - \epsilon > T-\dist(x_0,V)-2\epsilon-\epsilon > |t_0|.
\end{eqnarray*}
In particular, this yields $u(x_0,t_0)=0$.
\end{proof}

\section{Inverse problem for the heat kernel} \label{sec-IP}

Using the unique continuation result that we established, we can prove the uniqueness of the inverse problem for the heat kernel.
First, solving the inverse problem for the heat kernel is equivalent to solving an inverse spectral problem in our setting.
Let $\lambda_j,\, \phi_j$ ($j=0,1,\cdots$) be the eigenvalues and orthonormalized eigenfunctions of $-\Delta_X$. In particular, $\lambda_0=0$ and $\phi_0=\meas(X)^{-1/2}>0$.

\begin{lemma}\label{lem:heat eigen}
Let $(X,\dist,\meas)$ be a compact {\rm RCD}$(K,N)$ space as assumed in Theorem \ref{thm:unique almost}.
Suppose that we are given an open set $V\subset X$ and the measure $\meas$ on $V$.
Then the heat kernel $p$ on $V\times V\times \mathbb{R}_+$ determines the spectral data $\{\lambda_j,\phi_j|_V\}_{j\in \mathbb{N}}$ on $V$, where $\{\phi_j\}_{j\in \mathbb{N}}$ is a choice of orthonormalized eigenfunctions in $L^2(X,\meas)$.
\end{lemma}

\begin{proof}
The proof essentially follows \cite{KKL,KKLM}.
Recall (\ref{eq:heat exp}), namely
\begin{equation}
p(x,y,t)=\sum_{j=0}^{\infty} e^{-\lambda_j t} \phi_j (x) \phi_j(y),\quad \text{in $C(X \times X)$,}
\end{equation}
where $\{\phi_j\}_{j=0}^{\infty}$ is any complete family of orthonormalized eigenfunctions in $L^2(X,\meas)$.
For any $t>0$, using $\lambda_0=0$ and $\phi_0=\textrm{const}$, we consider
\begin{equation}\label{def-I0}
I_0(t):=\int_V p(x,x,t) \,\di\meas(x) = \meas(V) \phi_0^2  +\sum_{j=1}^{\infty} e^{-\lambda_j t} \int_V \phi_j^2 \di\meas.
\end{equation}
Observe that the latter term converges to zero as $t\to \infty$. Indeed, as $\phi_j$ is $L^2$-normalized with respect to $\meas$, for $t>1$,
\begin{eqnarray*}
\lim_{t\to\infty}\sum_{j=1}^{\infty} e^{-\lambda_j t} \int_V \phi_j^2 \di\meas &\leq & \lim_{J\to\infty} \lim_{t\to\infty}\Big(\sum_{j=1}^{J} e^{-\lambda_j t} + \sum_{j=J+1}^{\infty} e^{-\lambda_j t}\Big)  \\
&\leq&\lim_{J\to\infty} \lim_{t\to\infty} \Big(\sum_{j=1}^{J} e^{-\lambda_j t} + \sum_{j=J+1}^{\infty} e^{-\lambda_j} \Big) =0,
\end{eqnarray*}
where we have used a Weyl-type bound (\ref{eq:eigenfunction-estimate}). 
Since the set $V$ and the measure $\meas$ on $V$ are given, the quantity $\lim\limits_{t\to \infty} I_0(t)=\meas(V) \phi_0^2$
determines the (positive) constant $\phi_0$.

For the determination of eigenvalues $\lambda_j$, we first note that by unique continuation for eigenfunctions, which is a consequence of Theorem \ref{theorem-uc-general} (with Corollary \ref{theorem-uc-weak}), there does not exist an eigenfunction vanishing identically on $V$, which shows that all eigenvalues $\lambda_j$ appear in the summation \eqref{def-I0}.
To determine the eigenvalue $\lambda_1$, we can test for the unique number $\lambda \in \R$ such that
\begin{equation}
\lim_{t\to\infty} e^{\lambda t}\Big(I_0(t)-\meas(V) \phi_0^2 \Big) \in (0,\infty).
\end{equation}
Indeed, if $\lambda<\lambda_1$, then using \eqref{def-I0} and the same argument as above,
$$e^{\lambda t}\Big(I_0(t)-\meas(V) \phi_0^2 \Big)=\sum_{j=1}^{\infty} e^{(\lambda-\lambda_j) t} \int_V \phi_j^2 \di\meas \;\to 0,\quad \textrm{ for }t\to \infty.$$
If $\lambda>\lambda_1$, then the limit does not exist as $t\to \infty$ since the first term in the sum goes to infinity.
This determines the eigenvalue $\lambda=\lambda_1$.
Evaluating the limit gives
\begin{equation}
\lim_{t\to\infty} e^{\lambda_1 t}\Big(I_0(t)-\meas(V) \phi_0^2 \Big) = \sum_{\lambda_k=\lambda_1} \int_V \phi_k^2.
\end{equation}
Iterating this process determines the set of all eigenvalues $\lambda_j$.

\smallskip
Using the eigenvalues $\lambda_j$, for any fixed $x,y\in V$, we determine 
\begin{equation}
\sum_{\lambda_k=\lambda_1} \phi_k(x) \phi_k(y) = \lim_{t\to\infty} e^{\lambda_1 t} \Big(p(x,y,t)-\phi_0^2 \Big),
\end{equation}
and iteratively determine
\begin{equation} \label{eq-def-Qj}
Q_j(x,y):=\sum_{\lambda_{k}=\lambda_j}  \phi_{k}(x) \phi_{k}(y),\quad \textrm{ for any }x,y\in V, \textrm{ and all }j\in \mathbb{N},
\end{equation}
where we rely on (\ref{eq:eigenfunction-estimate}). 
When the multiplicity of an eigenvalue is 1, \eqref{eq-def-Qj} determines the eigenfunction up to sign.
In the general case of multiplicity large than 1, by \cite[Lem. 4.9]{KKL}, \eqref{eq-def-Qj} determines the multiplicity of eigenvalues and the orthonormalized eigenfunctions up to orthogonal transformations in eigenspaces.
Let us explain this argument below for the convenience of readers.

\smallskip
Consider the operator
$$(Q_j f)(x) := \int_V Q_j(x,y) f(y) \,\di\meas(y) = \sum_k \langle f, \phi_k \rangle_{L^2(V,\meas)} \phi_k(x), \quad x\in V.$$
As $\{\phi_k|_V\}$ are linearly independent due to Theorem \ref{theorem-uc-general},
the rank of operator $Q_j$ is the multiplicity of eigenvalue $\lambda_j$.
Since $Q_j(x,y),V,\meas|_V$ are known, 
the operator $Q_j$ and its image are known, which determines the rank or multiplicity of eigenvalues. Then we take a basis $v_1,\cdots, v_m$ in the range of $Q_j$ with rank $m$. We know that there exists an invertible matrix $A=(a_{ij})$ such that $\phi_k=\sum_l a_{kl} v_l$, for $k=1,\cdots,m$. Then
\begin{eqnarray*}
Q_j(x,y)&=&(\phi_1(x),\cdots,\phi_m (x)) (\phi_1(y),\cdots,\phi_m (y))^T \\
&=& (v_1(x),\cdots,v_m(x)) A^T A (v_1(y),\cdots,v_m(y))^T.
\end{eqnarray*}
Let us fix the orthonormalized eigenfunctions $\phi_k$ for the moment.
We show that with an appropriate choice of basis $\{v_k\}_{k=1}^m$, the knowledge of $Q_j(x,y)$, for all $x,y \in V$, determines $A^T A$ corresponding to the basis. 
To see this, we need to pick appropriate points $x,y\in V$.
We claim that one can choose a basis $v_1,\cdots,v_m$ and points $x_1,\cdots,x_m\in V$ such that 
\begin{equation} \label{basis-vkxl}
v_k(x_l)=\delta_{kl},\quad \textrm{ for }k,l=1,\cdots,m.
\end{equation}
Indeed, starting with $v_1$, we pick arbitrarily $v_1\neq 0$ and one point $x_1$ such that $v_1(x_1)\neq 0$, then normalize it to $v_1(x_1)=1$. Then we pick any $v_2\neq 0$ linearly independent to $v_1$. If $v_2(x_1)=0$, then pick another point $x_2$ such that $v_2(x_2)=1$ after normalization. Otherwise if $v_2(x_1)\neq 0$, then subtract appropriate multiple of $v_1$ from $v_2$ to define a new vector $\tilde{v}_2$ such that it vanishes at $x_1$, thus reducing to the former case. Repeating this procedure, we get an upper triangular matrix $v_k(x_l)$ with diagonal entries $1$. Then one can eliminate the other entries to make it the identity matrix \eqref{basis-vkxl}.
Using this specific choice of basis and points in \eqref{basis-vkxl}, we can determine $A^T A$ by
$$Q_j(x_k,x_l)= (\delta_{ki}) A^T A (\delta_{il})=(A^T A)_{kl},\quad \textrm{ for }k,l=1,\cdots,m.$$

Now knowing $A^T A$, we use the polar decomposition for $A$: $A=UP$, where $P=(A^T A)^{1/2}$ is known and $U$ is an unknown orthogonal matrix. Then
$$(\phi_1,\cdots,\phi_m)^T=A(v_1,\cdots,v_m)^T=U P (v_1,\cdots,v_m)^T.$$
Hence the given data determines $P (v_1,\cdots,v_m)^T$ on $V$, that is, the values of orthonormalized eigenfunctions, up to an (unknown) orthogonal matrix $U$, restricted on $V$.
\end{proof}

From the spectral data we can perform the reconstruction through a non-smooth version of the geometric Boundary Control method.
The first step is to prove the approximate controllability of the wave equation.
We consider the wave equation
\begin{equation} \label{eq-uf}
\left\{ \begin{aligned}
&(\partial_t^2-\Delta_X) u^f(x,t)=f(x,t), \quad {\rm supp}(f)\subset V \times \mathbb{R}_+,\\
& u^f |_{t=0}=\partial_t u^f |_{t=0}=0.
\end{aligned} \right.
\end{equation}
We denote by $u^f$ the weak solution of the wave equation with source $f$ and zero initial conditions.

\begin{lemma}
\label{prop-control}
Let $(X,\dist,\meas)$ be a compact {\rm RCD}$(K,N)$ space as assumed in Theorem \ref{thm:unique almost}.
Let $\tau>0$ and $V\subset \mathcal{R}$ be an open subset.
Then the set $\{u^f(\cdot,\tau): f\in C_c(V\times (0,\tau))\} $
is dense in $L^2(X(V,\tau))$, where 
$X(V,\tau):=\{x\in X: \dist(x,V)<\tau\}$
is called the domain of influence. 
\end{lemma}

\begin{proof}
We use a standard duality argument from unique continuation. 
Due to Theorem \ref{prop-finite-propagation}, we see that $\textrm{supp}(u^f(\cdot,\tau))\subset X(V,\tau)$.
Let $\psi \in L^2(X(V,\tau))$ be such that
$$\langle u^f(\cdot,\tau), \psi \rangle_{L^2(X)}=0, \quad \forall \, f\in C_c(V\times (0,\tau)).$$
To prove the claim it suffices to prove that $\psi=0$.

Let us consider a wave equation solving backward from $t=\tau$:
\begin{equation} \label{eq-w}
(\partial_t^2-\Delta_X)w=0,\quad w|_{t=\tau}=0,\quad \partial_t w|_{t=\tau}=\psi.
\end{equation}
Since we know $w \in E([0,T];X)$ and $w|_{t=\tau}=0$, we use the weak formulation of the wave equation \eqref{eq-uf} with test function $w$,
$$\int_0^{\tau} \int_X \Big(-\partial_t u^f\, \partial_t w+ \langle \nabla u^f,\nabla w\rangle \Big) \di\meas \di t= \int_0^{\tau} \int_X f w.$$
On the other hand, since $u^{f}|_{t=0}=0$, we use the weak formulation of the wave equation \eqref{eq-w} with test function $u^f$,
$$\int_0^{\tau} \int_X \Big(-\partial_t w\, \partial_t u^f+ \langle \nabla w,\nabla u^f\rangle \Big) \di\meas \di t= \int_X  u^f|_{t=\tau} \psi.$$
Thus,
$$\int_0^{\tau} \int_X f w = \int_X  u^f|_{t=\tau} \psi=0,$$
for all $f\in C_c(V\times (0,\tau))$. 
This implies that $w=0$ on $V\times (0,\tau)$.

To use the unique continuation result,
we consider the weak solution $\bar{w}$ of the wave equation $(\partial_t^2-\Delta_X) \bar{w}=0$ on $X\times [\tau,2\tau]$ with initial conditions $\bar{w}|_{t=\tau}=0$ and $\partial_t \bar{w}|_{t=\tau}=\psi$.
Then we consider 
\begin{equation} 
w_e=\left\{ \begin{aligned}
& w , \quad \textrm{for } t\in [0, \tau], \\
& \bar{w} , \quad \textrm{for } t\in (\tau,2\tau]. \\
\end{aligned} \right.
\end{equation}
By linearity of Laplacian, we see that $w_e(2\tau-t)=-w_e(t)$, so $w_e=0$ on $V\times (0,2\tau)$.
Since $\partial_t w$ exists in the strong sense for a.e. $t\in [0,\tau]$, then for any $\delta\in (0,\tau)$, there exists $\delta'\in [0,\delta)$ such that $\partial_t w|_{t=\delta'}$ exists in the strong sense.
Then by the proof of Corollary \ref{coro-wave-extension}, $w_e\in E([\delta',2\tau];X)$ is a weak solution of the wave equation on $X\times [\delta',2\tau]$ with initial conditions at $t=\delta'$.
Using the unique continuation Theorem \ref{theorem-uc-general} and Corollary \ref{theorem-uc-weak}, we have $w_e=0$ on $\{(x,t): \dist(x,V)<\tau-\delta'-|t-\tau| \}$. 
Letting $\delta$ approaching zero, we have $w_e=0$ on $\{(x,t): \dist (x,V)<\tau-|t-\tau| \}$ and hence $\psi=0$.
\end{proof}

Recall that the local (interior) distance function $r_x:V\to \mathbb{R}$ associated to $x\in X$ is defined by 
\begin{equation}
r_x(z) :=\dist (x,z),
\end{equation}
and denote the set of local distance functions by $R_V(X):=\{r_x: x\in X\}$, see e.g. \cite[Sec. 3.7]{KKL}.

\begin{lemma}
\label{prop-distance}
Let $(X,\dist,\meas)$ be a compact {\rm RCD}$(K,N)$ space as assumed in Theorem \ref{thm:unique almost}.
Suppose that we are given an open subset $V\subset \mathcal{R}$ and the metric $\dist$ and measure $\meas$ on $V$. 
Then the heat kernel on $V\times V\times \mathbb{R}_+$ determines the set of local distance functions $R_V(X)$ on $V$.
\end{lemma}

\begin{proof}
By Lemma \ref{lem:heat eigen}, the heat kernel on $V\times V\times \mathbb{R}_+$ determines the spectral data $\{\lambda_j,\phi_j|_V\}$ on $V$.
Given a source $f\in C_c(V\times [0,T])$, the spectral data on $V$ determine the Fourier coefficients $\langle u^f(\cdot,t) ,\phi_j \rangle$ at any time $t$ with zero initial conditions. 
This can be argued using the approximation solution $u^J\in C^2([0,T];D(\Delta_X))$ constructed in the proof of Proposition \ref{prop-regularity}, satisfying $u^J\to u^f$ in $L^{\infty}([0,T];H^{1,2}(X))$ as $J\to\infty$, with initial conditions $\psi_0^J\to 0$ in $H^{1,2}(X)$, $\psi_1^J\to 0$ in $L^2(X)$, and source $f^J\to f$ in $L^2(X\times [0,T])$.
Consider the Fourier coefficients 
$$u^J_j(t)=\int_X u^J(\cdot,t) \phi_j \,\di\meas\in C^2([0,T]),$$ 
which satisfy
$$\partial_t^2 u^J_j = \int_X (\Delta_X  u^J+f^J)\phi_j \,\di\meas=\int_X u^J \Delta_X \phi_j +\int_X f^J \phi_j =-\lambda_j u^J_j +\int_X f^J\phi_j\,\di\meas,$$
with initial conditions $u^J_j(0)=\int_X \psi_0^J \phi_j$ and $\partial_t u^J_j(0)=\int_X \psi_1^J \phi_j$.
Hence $u^J_j(t)$ has the formula:
$$u^J_j(t)=u^J_j(0) \cos(\sqrt{\lambda_j}t)+\frac{\partial_t u^J_j(0)}{\sqrt{\lambda_j}} \sin(\sqrt{\lambda_j}t) +\int_0^t \int_X \frac{\sin(\sqrt{\lambda_j}(t-\tau))}{\sqrt{\lambda_j}} f^J(\cdot,\tau) \phi_j \,\di\meas \di\tau.$$
As $u^J_j(0)\to 0$ and $\partial_t u^J_j(0)\to 0$ as $J\to \infty$, we have
\begin{equation} \label{eq-Blag}
u^f_j(t)=\int_0^t \int_X \frac{\sin(\sqrt{\lambda_j}(t-\tau))}{\sqrt{\lambda_j}} f(\cdot,\tau) \phi_j \,\di\meas \di\tau.
\end{equation}
Since $\textrm{supp}(f)\subset V\times \mathbb{R}$ and $\phi_j$ is known on $V$, the Fourier coefficients $u^f_j(t)$ of $u^f(\cdot,t)$ can be determined for any given $f\in C_c(V\times [0,T])$ for all $t$.

From here, one can follow the standard geometric Boundary Control method, see e.g. \cite{KKL,KL06,KLLY}, to determine the set of local distance functions $R_V(X)$ on $V$. 
This part only uses the formula \eqref{eq-Blag} and the approximate controllability of wave equation (Lemma \ref{prop-control}).
\end{proof}

\begin{lemma}\label{lemma-map-isometry-general}
Let $X_i$ be pointed $\RCD(K,N)$ spaces $(i=1,2)$, and for $x_i\in X_i$, let
$\psi:B_r(x_1) \to B_r(x_2)$
be a map that preserves the heat kernel:
$$p_{X_1}(x,y,t)=p_{X_2}(\psi(x),\psi(y),t), \quad \textrm{for all }x, y\in B_r(x_1),\, t\in \mathbb{R}_+.$$
If $\psi$ is surjective, then $\psi$ is a measure-preserving isometry.
\end{lemma}

\begin{proof}
By Varadhan's asymptotics, we know that $\psi$ preserves the distances.
In particular, $\psi$ gives an isometry as metric spaces, therefore for a point $\bar x \in B_r(x_1)$, $\bar x$ is $n$-regular if and only if so is $\psi (\bar x)$, where the regular set we are discussing here is a metric one using (metric) tangent cones. 

Finally let us check that $\psi$ preserves the measures.  To this end, denote by $n$ the essential dimensions of $X_i$, which coincide because of the arguments above, and fix an $n$-regular point $\bar x \in B_r(x_1)$. Letting $t \to 0^+$ in
\begin{equation}
    t^{\frac{n}{2}}p_{X_1}(\bar x, \bar x, t)=t^{\frac{n}{2}}p_{X_2}(\psi (\bar x), \psi (\bar x), t) 
\end{equation}
yields that $\bar x \in \mathcal{R}_{X_1, n}^*$ holds (recall (5) of Theorem \ref{thm:rcd structure}) if and only if $\psi(\bar x) \in \mathcal{R}_{X_2,n}^*$ holds, 
where we used a convergence result for the heat kernel proved in \cite{AHT} (see (\ref{eq:heat conv})) for an $n$-regular point $\bar x$;
\begin{equation} \label{short-time-diagonal}
    \meas_{X_1}(B_{\sqrt{t}}(\bar x))p_{X_1}(\bar x, \bar x, t) \to c_n \in (0, \infty).
\end{equation}
Recalling (5) of Theorem \ref{thm:rcd structure}, 
the above observation tells us
\begin{equation}
    \frac{\di\meas_{X_1}}{\di\mathcal{H}^n}(z)=\frac{\di\meas_{X_2}}{\di\mathcal{H}^n}(\psi(z)),\quad \text{for $\meas_{X_1}$-a.e. $z \in \mathcal{R}_{X_1, n}^*$}.
\end{equation}
In particular $\psi$ preserves the measures. 
\end{proof}

Under the setting of Theorem \ref{thm:unique almost}, we show that the surjectivity assumption for $\psi$ in the lemma above can be removed. In order to state it, we prepare the following, where roughly speaking, $x \in \mathcal{R}_{n, \epsilon}$ holds if and only if any tangent cone at $x$ is pGH-close to $\mathbb{R}^n$. 
\begin{definition}[Almost regular point]
    Let $(X, \dist, \meas)$ be an $\RCD(K, N)$ space  for some $K \in \mathbb{R}$ and some $N \in [1, \infty)$, and let $r \in (0,\infty), \epsilon \in (0, 1)$ and let $n \in \mathbb{N}$. Define the set, denoted by $\mathcal{R}_{X, n, \epsilon, r}$ or $\mathcal{R}_{n, \epsilon, r}$ for short, of points $x \in X$ with
    \begin{equation}
        \dist_{\mathrm{pGH}}\left((X,s^{-1}\dist_X, x), (\mathbb{R}^n, \dist_{\mathbb{R}^n}, 0_n)\right) \le \epsilon s,\quad \text{for all $s \in (0, r]$,}
    \end{equation}
    where recall that $\dist_{\mathrm{pGH}}$ denotes the pointed Gromov-Hausdorff distance. Finally, put
    \begin{equation}
        \mathcal{R}_{n, \epsilon}:=\bigcup_{r>0}\mathcal{R}_{n, \epsilon, r}\,.
    \end{equation}
\end{definition}

Let us note that pGH is enough in the definition above, not necessary pmGH, because of the splitting theorem, as mentioned around (\ref{eq:regular}).
Note that 
\begin{equation}
    \mathcal{R}_n=\bigcap_{\epsilon>0}\bigcup_{r>0}\mathcal{R}_{n, \epsilon, r}\, ,
\end{equation}
and that if $X$ is non-collapsed, then for any $\epsilon>0$, any $x \in \mathcal{R}_N$ is an interior point of $\mathcal{R}_{n, \epsilon, r}$ for some $r>0$. 

\begin{lemma}\label{lem:surjective}
For all $K \in \mathbb{R}$ and $N \in [1, \infty)$, there exists $\epsilon_{K,N}>0$ such that the following holds. 
Let $X_i$ be $\RCD(K,N)$ spaces for some $K \in \mathbb{R}$ and some $N \in [1, \infty)$, and let $V_i\subset X_i$ be open subsets $(i=1,2)$, and let
$\Phi:V_1 \to V_2$
be a map that preserves the heat kernel:
$$p_{X_1}(x,y,t)=p_{X_2}(\Phi(x),\Phi(y),t), \quad \textrm{for all }x,y\in V_1,\, t\in \mathbb{R}_+.$$    
Assume that $V_1 \cap \mathcal{R}_{n_1, \epsilon_{K, N}}$ has an interior point, where $n_1$ denotes the essential dimension of $X_1$.
Then there exist open subsets $B_1\subset V_1$ and $B_2\subset V_2$ such that
\begin{equation}\label{map-isometry2}
\Phi|_{B_1}: B_1\to B_2 \;\textrm{ is a measure-preserving isometry.}
\end{equation}
\end{lemma}
\begin{proof}
Let us divide the proof into the following $3$ steps.

\smallskip
\textbf{Step 1}: \textit{Let $\phi:Y \to Z$ be a map satisfying that the map $\Phi=(\mathrm{id}_Y, \phi):Y \to Y \times Z$
preserves the heat kernels, where both $Y, Z$ are $\RCD(0,N)$ spaces for some $N \in [1, \infty)$. Then $Z$ is a single point.}

\smallskip
Since the heat kernel of $Y \times Z$ also splits in the sense;
\begin{equation}
    p_{Y \times Z}((y_1, z_1), (y_2, z_2), t)=p_Y(y_1, y_2, t)\cdot p_Z(z_1, z_2, t), \quad y_i \in Y, z_i \in Z,
\end{equation}
our assumption yields that
$p_{Z}(z, z, t)=1$ holds for any $t>0$, where this does not happen if $Z$ is not a single point. This is due to the fact that in general the order of $p(x,x,t)$ as $t \to 0^+$ is equal to that of $\meas(B_{\sqrt{t}}(x))^{-1}$. Moreover if $X$ is not a single point, then $\meas(B_r(x)) \le Cr$ for any $r \in (0, 1]$. Thus we get \textit{Step 1}.

\smallskip
\textbf{Step 2}: \textit{For all $K \in \mathbb{R}$ and $N \in [1, \infty)$, there exists $\epsilon_{K, N}>0$ such that the following hold; let $(X, \dist_X, \meas_X, x)$ be a pointed $\RCD(K, N)$ space, and let $(Y, \dist_Y, \meas_Y, y)$ be a pointed $\RCD(-\epsilon_{K,N}, N)$ space. If 
\begin{enumerate}
    \item there exists $r_x \in (0, 1]$ such that for any $\bar x \in \bar B_{r_x}(x)$, there exists $r_{\bar x} \in (0, 1]$ such that 
\begin{equation}
    \dist_{\mathrm{pGH}}\left( \left( X, s^{-1}\dist_X, \bar x\right), \left(\mathbb{R}^n, \dist_{\mathbb{R}^n}, 0_n)\right)\right)\le \epsilon_{K, N}, \quad \textit{for any $s \in (0, r_{\bar x}]$;}
\end{equation}
    \item there exists a map $\phi:(X, x) \to (Y, y)$ preserving the heat kernels,
\end{enumerate}
then $\phi(\bar B_{\frac{r_x}{100}}(x))=\bar B_{\frac{r_x}{100}}(y)$.}

\smallskip
Before starting the proof, notice by the proof of Lemma \ref{lemma-map-isometry-general} that if a map $\phi:(X, \dist_X, \meas_X) \to (Y, \dist_Y, \meas_Y)$ between RCD spaces preserves the heat kernels, then so is $\phi:(X, s^{-1}\dist_X, c\meas_X) \to (Y, s^{-1}\dist_Y, c\meas_Y)$ for all $s>0$ and $c>0$, and $\meas_X(B_r(x))^{-1}\meas_Y(B_r(\phi(x))) \to 1$ as $r \to 0^+$ for any $x \in X$, where this fact will be used immediately below.

The proof is done by contradiction; Assume that there exist;
\begin{itemize}
\item a sequence $\epsilon_i \to 0^+$;
    \item sequences of pointed $\RCD(K, N)$ spaces $(X_i, \dist_{X_i}, \meas_{X_i}, x_i)$ and of $r_{x_i} \to 0^+$ satisfying that for any $\bar x \in \bar B_{r_i}(x_i)$, there exists $r_{\bar x} \in (0, 1]$ such that 
    \begin{equation}
        \dist_{\mathrm{pGH}}\left( (X_i, s^{-1}\dist_{X_i}, \bar x_i), (\mathbb{R}^n, \dist_{\mathbb{R}^n}, 0_n)\right) \le \epsilon_i, \quad \text{for any $s \in (0, r_{\bar x_i}]$,} 
    \end{equation}
    for some $n \in \mathbb{N} \cap [1, N]$;
    \item sequences of pointed $\RCD(-\epsilon_i, N)$ spaces $(Y_i, \dist_{Y_i}, \meas_{Y_i}, y_i)$ and of maps $\phi_i:(X_i, x_i) \to (Y_i, y_i)$ preserving the heat kernels (thus, they also preserve the distances by Varadhan's asymptotics) such that there exists $\bar y_i \in \bar B_{\frac{r_{x_i}}{100}}(y_i) \setminus \phi_i(\bar B_{\frac{r_{x_i}}{100}}(x_i))$.
\end{itemize}
Notice that we can easily see $\bar y_i \not \in \phi_i(\bar B_{r_{x_i}}(x_i))$ (by contradiction). Find $\bar z_i \in \phi_i(\bar B_{r_{x_i}}(x_i))$, thus $\bar z_i=\phi_i(\bar x_i)$ for some $\bar x_i \in \bar B_{r_{x_i}}(x_i)$, with
\begin{equation}
    s_i:=\dist_{Y_i}(\bar y_i, \bar z_i)=\dist_{Y_i}(\bar y_i, \phi_i(\bar B_{r_{x_i}}(x_i))) \in (0, r_i].
\end{equation}
Note that the triangle inequality shows $\bar x_i \in \bar B_{\frac{r_{x_i}}{2}}(x_i)$. Put $t_i:=\min \{s_i, r_{\bar x_i}\}$ and consider rescaled RCD spaces:
\begin{equation}
    \left(X_i, t_i^{-1}\dist_{X_i}, \meas_{X_i}(B_{t_i}(\bar x_i))^{-1}\meas_{X_i}, \bar x_i\right), \quad \left( Y_i, t_i^{-1}\dist_{Y_i}, \meas_{X_i}(B_{t_i}(\bar x_i))^{-1}\meas_{Y_i},\bar z_i\right),
\end{equation}
where we notice a fact that 
    $
        \frac{\meas_{Y_i}(B_{t_i}(\bar z_i))}{\meas_{X_i}(B_{t_i}(\bar x_i))}
    $ 
    has uniform positive two-sided bounds. This fact follows from our assumption, $\phi_i$ preserves the heat kernel and a Gaussian estimate (on the diagonal part), more precisely;
    \begin{equation}
        \meas_{X_i}(B_{t_i}(\bar x_i))^{-1}\sim p_{X_i}(\bar x_i, \bar x_i, t_i^2)=p_{Y_i}(\bar z_i, \bar z_i, t_i^2)\sim \meas_{Y_i}(B_{t_i}(\bar z_i))^{-1}.
    \end{equation}
Letting $i \to \infty$ after passing to a subsequence, together with the splitting theorem \cite{G13} and the continuity of the heat kernels with respect to the pmGH convergence (\ref{eq:heat conv}), 
allows us to find a map 
\begin{equation}
    \mathbb{R}^n \to \mathbb{R}^n\times W, \quad v \mapsto (v, \phi(v))
\end{equation}
preserves the heat kernels, where $W$ is an $\RCD(0, N)$ space which is not a single point (e.g. the proofs of \cite[Prop. 4.1 and Th. 4.3]{H1dim}). This contradicts \textit{Step 1}. Thus we have \textit{Step 2}.

\smallskip
\textbf{Step 3}: \textit{Conclusion.}

\smallskip
Fixing $\epsilon_{k, n}$ as appeared in \textit{Step 2}, fix an interior point $x$ of $V_1 \cap \mathcal{R}_{n_1, \epsilon_{k, N}}$. Then applying Lemma \ref{lemma-map-isometry-general} together with \textit{Step 2} to a small ball centered at $x$ completes the proof.
\end{proof}

\begin{remark}
Note that Lemma \ref{lem:surjective} is valid under the assumptions of Theorem \ref{thm:unique almost} and, in particular, the essential dimensions of $X_i$ coincide: $n_1=n_2$.
In the special case of $X_1$ being a weighted Lipschitz Riemannian manifold (in particular $\mathcal{R}=X_1$), the argument for Lemma \ref{lem:surjective} is simplified as follows.
The preservation of the heat kernels implies that $\frac{\meas_{X_2}(B_r(\Phi(x)))}{\meas_{X_1}(B_r(x))}$ has uniform two-sided bounds due to the Gaussian estimate \cite{JLZ} (Theorem \ref{thm:gaussian}), and thus the dimensions must coincide. Then the distance preservation, which is justified by Varadhan's asymptotics, implies that $\Phi(B_r(x))$ for small $r$ is open in $X_2$ by the invariance of domain. Hence $\Phi|_{B_r(x)}$ is a measure-preserving isometry for small $r$ by the short-time diagonal asymptotics of the heat kernel \eqref{short-time-diagonal}. 
\end{remark}

\begin{proposition} \label{isometry-homeo}
For $i=1,2$, let $X_i$ be two compact $\RCD(K,N)$ spaces in the class specified in Theorem \ref{thm:unique almost}.
Let $V_i\subset X_i$ be open subsets. 
Suppose that there is a map $\psi:V_1\to V_2$, such that 
$$p_{X_1}(x,y,t)=p_{X_2}(\psi(x),\psi(y),t),\quad \textrm{for all }x,y\in V_1,\, t\in \mathbb{R}_+.$$
Then one can construct a homeomorphism $\Psi:X_1\to X_2$ such that $\Psi|_{V_1}=\psi$.
\end{proposition}

\begin{proof}
By Varadhan's asymptotics, $\psi$ is distance-preserving on $V_1$. By Lemma \ref{lem:surjective}, there exist open subsets $B_1\subset V_1$ and $B_2\subset V_2$ such that
\begin{equation}\label{map-isometry}
\psi|_{B_1}: B_1\to B_2 \;\textrm{ is a measure-preserving isometry.}
\end{equation}
Without loss of generality, we can assume that $B_1\subset \mathcal{R}_{X_1}$ and $B_2\subset \mathcal{R}_{X_2}$; otherwise, we can restrict the domain of $\psi$ to $B_1\cap \mathcal{R}_{X_1} \cap \psi^{-1}(B_2\cap \mathcal{R}_{X_2})$.
In particular, it follows that $\psi|_{B_1}$ preserves the Lipschitz Riemannian metrics due to \cite{CH70,Taylor06}.

For $i=1,2$, we claim that the map $R_i:X_i\to L^{\infty}(B_i)$, defined by 
$$R_i(x)= r_x|_{B_i},$$ 
is a homeomorphism onto its image. This is because the map $R_i$ is continuous and injective. 
Indeed,
equipped $R_i(X_i)$ with $L^{\infty}(B_i)$-norm, it follows from the triangle inequality that the map $x\mapsto r_x|_{B_i}$ is continuous.
The injectivity is due to \cite[Th. 4]{Ye} and geodesic non-branching in RCD spaces \cite{D}, which is stated in Theorem \ref{thm:rcd structure}(1).
For the convenience of readers, we include a short argument below in our particular case of $B_i$ being open.
Suppose there exist two points $x,y\in X_1$ such that $r_x=r_y$ on $B_1$ but $x\neq y$. 
Let $\bar{B}\subset B_1$ be a closed metric ball centered at $p$ not containing $x,y$.
Let $z\in \partial B$ be a point such that $\dist_{X_1}(x,z)=\dist_{X_1}(x,\partial B)$, and the same is also satisfied by $y$ due to $r_x=r_y$.
Let $[zx], [zy]$ be a shortest path from $z$ to $x,y$ respectively. Taking any shortest path $[pz]$ from the center $p$ to $z$, the paths $[pz]\cup [zx]$ and $[pz]\cup [zy]$ are both shortest paths from $p$ that agree on $[pz]$. Thus $x=y$ due to geodesic non-branching.

Since $\psi|_{B_1}:B_1\to B_2$ is a measure-preserving isometry that preserves the heat kernel by assumption, the local distance functions $R_i(X_i)$ on $B_i$ coincide due to Lemma \ref{prop-distance}.
Identifying $B_2$ with $B_1$ via $\psi$,
we can define $\Psi:X_1\to X_2$ by 
\begin{equation} \label{def-homeo}
\Psi:= R_2^{-1}\circ R_1,
\end{equation}
which is a homeomorphism, since $R_i$ is a homeomorphism onto its image.

It remains to show that $\Psi|_{V_1}=\psi$.
Suppose that there exists a point $q\in V_1$ such that $\Psi(q)\neq \psi(q)$. Let us take an open metric ball $B'\subset B_1$ such that its closure does not contain $q$.
By definition \eqref{def-homeo}, $r_q=r_{\Psi(q)}\circ \psi$ on $B_1$.
Then for an arbitrary point $z\in B'$, since $\psi$ is distance-preserving on $V_1$, we have
$$\dist_{X_1}(q,z)=\dist_{X_2}(\psi(q),\psi(z))=\dist_{X_2}(\Psi(q),\psi(z)).$$
In particular, $\psi(q),\Psi(q)\notin \psi(B')$.
Ranging over all $z\in B'$, we see that $r_{\Psi(q)}=r_{\psi(q)}$ on an open subset $\psi(B') \subset B_2$.
Using the injectivity of $x\mapsto r_x|_{\psi(B')}$, we conclude that $\Psi(q)= \psi(q)$.
\end{proof}

\begin{proposition} \label{isometry-eigen}
For $i=1,2$, let $X_i$ be two compact $\RCD(K,N)$ spaces in the class specified in Theorem \ref{thm:unique almost}.
Let $V_i\subset X_i$ be open subsets. 
Suppose that there is a map $\psi:V_1\to V_2$, such that 
$$p_{X_1}(x,y,t)=p_{X_2}(\psi(x),\psi(y),t),\quad \textrm{for all }x,y\in V_1,\, t\in \mathbb{R}_+.$$
Let $\Psi:X_1\to X_2$ be the homeomorphism constructed in Proposition \ref{isometry-homeo}.
Then there exist complete families of orthonormalized eigenfunctions $\{\phi_{i;j}\}_{j\in \mathbb{N}}$ such that 
$$\phi_{1;j}=\phi_{2;j}\circ \Psi \; \textrm{ on }X_1, \quad \textrm{ for all }j,$$
where $\phi_{i;j}$ denotes the $j$-th eigenfunction of $X_i$.
\end{proposition}

\begin{proof}
From Proposition \ref{isometry-homeo}, we construct a homeomorphism $\Psi:X_1\to X_2$ such that $\Psi|_{V_1}=\psi$. In particular, the restriction $\psi:V_1\to \psi(V_1)$ is a homeomorphism for which $\psi(V_1)$ is open in $X_2$, and thus a measure-preserving isometry by the proof of Lemma \ref{lemma-map-isometry-general}. Without loss of generality, let us assume that $V_2=\psi(V_1)$.

By Lemma \ref{lem:heat eigen}, since $\psi$ preserves the heat kernel, the spectral data on $V_i$ coincide (after identifying $V_2$ with $V_1$ via $\psi$).
Namely, the eigenvalues for the two spaces coincide, and
there exist complete families of orthonormalized eigenfunctions $\{\phi_{i;j}\}_{j\in \mathbb{N}}$, for $i=1,2$, such that
\begin{equation} \label{eigen-V}
\phi_{1;j}=\phi_{2;j} \circ \psi \; \textrm{ on }V_1, \quad \textrm{ for all }j.
\end{equation}
Using these spectral data, the formula \eqref{eq-Blag} and the approximate controllability of wave equation (Lemma \ref{prop-control}), \cite[Lem. 6.9]{KLLY} is valid in our setting:
\begin{equation} \label{eq-integral-eigen}
\int_{I_1\subset X_1} \phi_{1;j}\, \phi_{1;m} \,\di\meas_{X_1} = \int_{I_2\subset X_2} \phi_{2;j} \,\phi_{2;m} \,\di\meas_{X_2}, \quad \textrm{for all }j, m,
\end{equation}
where the sets $I_i$ are subsets of $X_i$ of the form
\begin{equation} \label{def-I12}
I_1=\bigcap_{l=1}^L \Big(X(U_{2l-1},s_{2l-1})\setminus X(U_{2l},s_{2l}) \Big), \quad I_2=\bigcap_{l=1}^L \Big(X(\psi(U_{2l-1}),s_{2l-1})\setminus X(\psi(U_{2l}),s_{2l}) \Big),
\end{equation}
for open subsets $U_l\subset V_1$ and $s_l\in \mathbb{R}$.
Recall that $X(V,\tau)=\{x\in X: \dist_X(x,V)<\tau\}$ is the domain of influence.
Taking $j=m=0$ in \eqref{eq-integral-eigen} so that $\phi_{1;0}=\phi_{2;0}=\textrm{const}>0$, then $\meas_{X_1}(I_1)=\meas_{X_2}(I_2)$.

Let $p\in X_1$ be arbitrary. Now we prove that
\begin{equation} \label{identify-eigen-p}
\phi_{1;j} (p) = \phi_{2;j}(\Psi(p)), \quad \textrm{ for all }j.
\end{equation}
The local distance function $r_p$ identifies with $r_{\Psi(p)}$ by definition \eqref{def-homeo}, that is,
\begin{equation} \label{identify-rx}
r_x= r_{\Psi(x)}\circ \psi \;\textrm { on }V_1, \quad\textrm{for any }x\in X_1.
\end{equation}
Because the map $x\mapsto r_x$ is injective, we can find sequences of sets $I_{1;k}\subset X_1$ and $I_{2;k}\subset X_2$, for $k\in \mathbb{N}_+$, such that
\begin{equation} \label{thin-slices-p}
\{p\}=\bigcap_{k} I_{1;k}, \quad \big\{\Psi(p) \big\}=\bigcap_{k} I_{2;k}.
\end{equation}
Indeed, one can define $I_{1;k}$ by taking $U_{2l-1}=U_{2l}$ in \eqref{def-I12} to be a small neighborhood (say, of radius $1/2k$) centered at each point $\xi_l$ of a finite $1/k$-net in $V_1$, and take $s_{2l-1}=r_p(\xi_l)+2/k$, $s_{2l}=r_p(\xi_l)-2/k$.
Observe that each $I_{1;k}$ contains a ball $B(p,1/2k)$, and $I_{1;k+1}\subset I_{1;k}$ for all $k$.
For $I_{2;k}$, one can take the domains $\psi(U_{2l})$ with the same numbers $s_{2l-1},s_{2l}$ due to \eqref{identify-rx}.
Then \eqref{eq-integral-eigen} yields $\meas_{X_1}(I_{1;k})=\meas_{X_2} (I_{2;k})>0$ for all $k$.
Now let $m=0$ and any $j\geq 1$,
\begin{equation}
\frac{1}{\meas_{X_1}(I_{1;k})} \int_{I_{1;k}} \phi_{1;j}\, \phi_{1;0} \,\di\meas_{X_1} = \frac{1}{\meas_{X_2}(I_{2;k})} \int_{I_{2;k}} \phi_{2;j} \,\phi_{2;0} \,\di\meas_{X_2}, \quad \textrm{ for all }k.
\end{equation}
By \eqref{thin-slices-p} and continuity of eigenfunctions, \eqref{identify-eigen-p} follows by taking $k\to \infty$.
\end{proof}

From Proposition \ref{isometry-eigen}, we can use Varadhan's asymptotics and short-time diagonal asymptotics in RCD spaces to prove the uniqueness result.

\begin{theorem}
\label{uniqueness-bijection}
For $i=1,2$, let $X_i$ be two compact $\RCD(K,N)$ spaces in the class specified in Theorem \ref{thm:unique almost}.
Let $V_i\subset X_i$ be open subsets.
Suppose that there is a map $\psi:V_1\to V_2$, such that 
$$p_{X_1}(x,y,t)=p_{X_2}(\psi(x),\psi(y),t),\quad \textrm{for all }x,y\in V_1,\, t\in \mathbb{R}_+.$$
Then there exists a unique measure-preserving isometry $\Psi:X_1\to X_2$ satisfying $\Psi|_{V_1}=\psi$. 
\end{theorem}

\begin{proof}
By Proposition \ref{isometry-homeo} and Lemma \ref{lemma-map-isometry-general}, it follows that the map $\psi:V_1\to \psi(V_1)$ is a measure-preserving isometry.
Then by Lemma \ref{lem:heat eigen}, the heat kernel on $V\times V\times \mathbb{R}_+$ determines the spectral data $\{\lambda_j,\phi_j|_V\}$ on $V$. Namely, the eigenvalues of the two spaces coincide, and
there exist complete families of orthonormalized eigenfunctions $\{\phi_{i;j}\}_{j\in \mathbb{N}}$, for $i=1,2$, such that
\begin{equation} \label{eigen-V}
\phi_{1;j}=\phi_{2;j} \circ \psi \; \textrm{ on }V_1, \quad \textrm{ for all }j,
\end{equation}
where $\phi_{i;j}$ denotes the $j$-th eigenfunction of $X_i$.
By Lemma \ref{prop-distance}, this identifies the set of local distance functions $R_V(X)$ on $V$.
Then by Proposition \ref{isometry-homeo}, this determines a homeomorphism 
\begin{equation}
\Psi:X_1\to X_2, \quad \textrm{with }\,\Psi|_{V_1}=\psi. 
\end{equation}
Then by Proposition \ref{isometry-eigen},
\begin{equation} \label{identify-eigen}
\phi_{1;j} (x) = \phi_{2;j}(\Psi(x)), \quad\textrm{for any }x\in X_1,
\end{equation}
where $\phi_{i;j}$ is the $j$-th orthonormalized eigenfunction of $X_i$ such that \eqref{eigen-V} holds.

Now from \eqref{identify-eigen} and the heat kernel expansion, we see that
\begin{equation}
p_{X_1}(x,y,t)=p_{X_2}(\Psi(x),\Psi(y),t),\quad \textrm{for any }x,y\in X_1, \; t\in \mathbb{R}_+.
\end{equation}
Then Varadhan’s asymptotics on RCD spaces yields that 
$$\dist_{X_1}(x,y)=\dist_{X_2}(\Psi(x),\Psi(y)),\quad \textrm{for any }x,y\in X,$$
so $\Psi:X_1\to X_2$ is actually an isometry. 
Again using \eqref{short-time-diagonal},
we have $\rho_1(x)=\rho_2(\Psi(x))$ for any regular point $x$. 
Thus $\Psi$ is a measure-preserving isometry. 

\smallskip
For the uniqueness of isometry $\Psi$, suppose that there exists another isometry $\Psi': X_1\to X_2$ satisfying $\Psi'|_{V_1}=\psi$. Assume that there exists $q\in X_1\setminus V_1$ such that $\Psi(q)\neq \Psi'(q)$. 
For an arbitrary point $z\in V_1$, by definition, 
$$\dist_{X_1}(q,z)=\dist_{X_2}(\Psi(q),\psi(z))=\dist_{X_2}(\Psi'(q),\psi(z)).$$
Ranging over all $z\in V_1$ hence all $\psi(z)\in \psi(V_1)$, we see
$$r_{\Psi(q)}|_{\psi(V_1)}=r_{\Psi'(q)}|_{\psi(V_1)}\quad \textrm{in }X_2.$$
This implies $\Psi(q)= \Psi'(q)$ since $x\mapsto r_x|_{\psi(V_1)}$ is injective (as $\psi(V_1)$ is open in $X_2$), which is a contradiction.
\end{proof}

\begin{remark} \label{remark-reconstruction}
Under the setting of Theorem \ref{thm:unique almost}, our method actually proves more than uniqueness: it constructs a metric-measure isometric copy of the original space $(X,\dist,\meas)$ on which the heat kernel data (or complete spectral data)
are given on an open subset $V\subset X$ (while the space is unknown outside of $V$). 
Namely, the geometric Boundary Control method constructs a topological space $(R_V(X),L^{\infty}(V))$ that is homeomorphic to $(X,\dist)$, and further determines the values of a complete family of orthnormalized eigenfunctions on $X$ and hence the heat kernel on $X\times X \times \mathbb{R}_+$.
Recall that $R_V(X)=\{r_x: x\in X\}$ is the set of local distance functions on $V$.
Then one can compute a metric $\widehat{\dist}$ on $R_V(X)$ via Varadhan's asymptotics such that $(R_V(X),\widehat{\dist})$ is isometric to $(X,\dist)$:
\begin{equation}
\widehat{\dist}(r_x,r_y)^2:=\dist(x,y)^2=-\lim_{t\to 0^+} 4t\log p(x,y,t), \quad \textrm{ for any }x,y\in X.
\end{equation}
For any regular point $x$ of $(R_V(X),\widehat{d})$, the short-time diagonal asymptotics \eqref{short-time-diagonal} determines the density 
\begin{equation}\label{eq-density-rx}
\widehat{\rho}(r_x):=\rho(x)=\lim_{t\to 0^+} \frac{c_n}{\omega_n t^{\frac{n}{2}} p(x,x,t)}, \quad \textrm{ for any $n$-regular point }x,
\end{equation}
where the dimension $n$ is the unique integer such that the limit in \eqref{eq-density-rx} belongs to $(0,+\infty)$.
This determines a measure $\widehat{\meas}$ such that $(R_V(X),\widehat{\dist},\widehat{\meas})$ is metric-measure isometric to $(X,\dist,\meas)$.
\end{remark}

\section{Continuity of Gel'fand's inverse problem under bounded Ricci curvature}\label{sec-stability}

In this section, using Theorem \ref{uniqueness-bijection} we prove the continuity of Gel’fand’s inverse problem.

\begin{definition}[Spectral approximation]\label{def:spec app}
Let $(X,x), (Y,y)$ be pointed compact $\RCD(K,N)$ spaces for some $K \in \mathbb{R}$ and some $N \in [1, \infty)$, and let 
\begin{equation}
\psi:(B_r(x), x) \to (B_r(y), y)
\end{equation}
be a map.
\begin{enumerate}
    \item{(Spectral heat-approximation)} $\psi$ is called an \textit{(spectral) $\epsilon$-approximation of the heat kernel}  if the following hold:
\begin{enumerate}
\item $\psi$ is an $\epsilon$-GH approximation;
\item we have
\begin{equation}\label{eq:heat comparison}
 1-\epsilon \le \frac{p_Y(\psi(x_1), \psi(x_2), t)}{p_X(x_1, x_2, t)} \le 1+\epsilon,
\end{equation}
holds for all $t \in [\epsilon, 1]$ 
and $x_1, x_2 \in B_r(x)$, where $p_X, p_Y$ denote the heat kernels of $X, Y$, respectively.
\end{enumerate}
    \item{(Spectral eigen-approximation)} $\psi$ is called an \textit{(spectral) $\epsilon$-approximation of the eigenfunctions} if the following hold:
\begin{enumerate}
\item $\psi$ is an $\epsilon$-GH approximation;
\item denoting by $k$ the integer part of $\frac{1}{\epsilon}$, there exist spectral data $\{\phi_{i}^X\}_{i=0}^{\infty}, \{\phi_{i}^Y\}_{i=0}^{\infty}$ of $X, Y$, respectively, such that 
\begin{equation}\label{eq:eigenfunction close}
    |\lambda_i^Y-\lambda_i^X|+\left| \phi_{i}^Y \circ \psi(z)-\phi_{i}^X(z)\right| <\epsilon, \quad \text{for any $z \in B_r(x)$ and any $1 \le i \le k$,}
\end{equation}
is satisfied.
\end{enumerate}
\end{enumerate}
\end{definition}
Let us prove that (a) in both (1) and (2) in the definition above can be actually removed.
\begin{proposition}[Weaker formulation of spectral approximation]\label{prop:weak}
Let $X, Y, \psi$ be as in the setting of Definition \ref{def:spec app}, with 
\begin{equation}\label{eq:diam}
        \mathrm{diam}(X), \diam(Y), \meas_X(X), \meas_Y(Y) \in [D^{-1}, D],\quad \text{for some $D>1$.}
    \end{equation}
    Then for some $\varepsilon=\varepsilon_{K, N, r, D}(\epsilon)$ (recall notation (\ref{eq:small})), 
we have the following. 
\begin{enumerate}
    \item If $B_r(y) \subset B_{\epsilon}(\psi(B_r(x)))$ holds, and (\ref{eq:heat comparison}) is satisfied for all $t \in [\frac{1}{2}, 1]$ and $x_1, x_2 \in B_r(x)$, then $\psi|_{B_{\frac{r}{4}}(x)}$ is an $\varepsilon$-approximation of the heat kernel and it is also an $\varepsilon$-mGH approximation;
    \item If (b) of (2) of Definition \ref{def:spec app} holds, then (\ref{eq:heat comparison}) is satisfied for all $t \in [\varepsilon, 1]$ and $x_1, x_2 \in B_r(x)$. 
\end{enumerate}
\end{proposition}
\begin{proof} First let us prove (1). The proof is divided into $3$ steps as follows, where in the sequel, $\varepsilon=\varepsilon_{K, N, r, D}(\epsilon)$.

\smallskip
\textbf{Step 1}:\textit{
    If (\ref{eq:heat comparison}) is satisfied for all $t \in [\frac{1}{2}, 1]$ and $x_1, x_2 \in B_r(x)$, then (\ref{eq:heat comparison}) holds for all $t \in [\varepsilon, 1]$ and $x_1, x_2 \in B_r(x)$.}

\smallskip
    This is done by observing the analyticity of the quotient 
    \begin{equation*}
        \frac{p_Y(\psi(x_1), \psi(x_2), t)}{p_X(x_1, x_2, t)}
    \end{equation*}
    in $t>0$ together with the Taylor expansion around $\frac{1}{2}$ of the radius $\frac{1}{2}-\varepsilon$.

\smallskip
    \textbf{Step 2}: \textit{If (\ref{eq:heat comparison}) is satisfied for all $t \in [\epsilon, 1]$ and $x_1, x_2 \in B_r(x)$, then 
    \begin{equation}\label{eq:almost bilip}
    -\varepsilon-(1+\varepsilon)\dist_X(x_1, x_2) \le \dist_Y(\psi(x_1), \psi(x_2)) \le (1+\varepsilon)\dist_X(x_1, x_2) +\varepsilon,
\end{equation}
for all $x_1, x_2 \in B_r(x)$.}

\smallskip
This is a direct consequence of a Gaussian estimate, Theorem \ref{thm:gaussian}, together with (\ref{eq:heat comparison}).

\smallskip
\textbf{Step 3}: \textit{If $B_r(y) \subset B_{\epsilon}(\psi(B_r(x)))$ holds, and (\ref{eq:heat comparison}) is satisfied for all $t \in [\epsilon, 1]$ and $x_1, x_2 \in B_r(x)$, then $\psi_{\bar B_r(x)}$ is an $\varepsilon$-measured Gromov-Hausdorff approximation, thus the proof of (1) is completed.}

\smallskip
The proof is done by a contradiction; If it is not the case, then there exist sequences of compact $\RCD(K, N)$ spaces $X_i, Y_i$ with (\ref{eq:diam}), $\tau>0$ and a sequence of maps $\psi_i:(B_r(x_i), x_i) \to (B_r(y_i), y_i)$ such that $\psi_i$ is an $\epsilon_i$-approximation of the heat kernel for some $\epsilon_i \to 0^+$ and that $\psi_i$ is not a $\tau$-measured Gromov-Hausdorff approximation. Thanks to the compactness of RCD spaces with respect to the mGH convergence, after passing to a subsequence, $(X_i, x_i), (Y_i, y_i)$ pmGH converge to compact $\RCD(K, N)$ spaces $(X,x), (Y, y)$, respectively. From now on let us construct a limit map $\psi:\bar B_r(x) \to \bar B_r(y)$ as follows.

First, take a countable dense subset $A$ of $\bar B_r(x)$ and for any $\bar x \in A$, find a sequence $\bar x_i \in B_r(x)$ converging to $\bar x$. Applying a diagonal argument with the compactness of $\bar B_r(y)$, after passing to a subsequence, with no loss of generality we can assume that $\psi_i(\bar x_i)$ converges to a point of $\bar B_r(y)$, which is denoted by $\psi(\bar x)$. Since $\psi:A \to \bar B_r(y)$ preserves the heat kernels, it also preserves the distances by Valadhan's asymptotics. Thus it has a unique continuous extension, still denoted by $\psi:\bar B_r(x) \to \bar B_r(y)$. It follows from the construction and our assumption that $\psi$ is a surjective distance-preserving map, namely an isometry as metric spaces. 
Then by Lemma \ref{lemma-map-isometry-general}, we know that $\psi$ is a metric measure isometry.
By the construction of $\psi$ together with \textit{Step 2}, we know that $\psi |_{B_r(x_i)}$ is a $\tau$-measured Gromov-Hausdorff approximation for sufficiently large $i$, which is a contradiction.

\smallskip
Next, let us prove (2).
Our assumption together with quantitative $L^{\infty}$-estimates on eigenfunctions (\ref{eq:eigenfunction-estimate}) yields
\begin{equation}\label{eq:heat kernel close}
    \left| p_X(x_1, x_2, t)-p_Y(\psi(x_1), \psi(x_2), t)\right| \le \varepsilon
\end{equation}
for all $t \in [\varepsilon, 1]$ and $x_1, x_2 \in B_r(x)$. This proves (2) because of a Gaussian estimate, Theorem \ref{thm:gaussian}. Thus we conclude.
\end{proof}

The following includes a quantitative converse of Lemma \ref{lem:heat eigen} in a general setting. 
\begin{proposition}[Equivalence between spectral and GH-approximations]\label{prop:equiv}
    Let $X, Y$ be pointed compact $\RCD(K, N)$ spaces for some $K \in \mathbb{R}$ and some $N \in [1, \infty)$ with (\ref{eq:diam}).
    Then for some $\varepsilon=\varepsilon_{K, N, r, D}(\epsilon)$ we have the following.
    \begin{enumerate}
        \item if a map $\psi: (B_r(x), x) \to (B_r(y), y)$ is an $\epsilon$-approximation of the eigenfunctions, then it is also an $\varepsilon$-approximation of the heat kernel;
        \item if a map $\psi:X \to Y$ is an $\epsilon$-mGH approximation, then $\psi$ is an $\varepsilon$-approximation of the eigenfunctions.
    \end{enumerate}
\end{proposition}
\begin{proof}
The proof of (1) is done by a contradiction as in the proof of Proposition \ref{prop:weak} again, though we can also provide a direct proof via the Gaussian estimate. 
For reader's convenience, let us give the details. 

If it is not the case, then there exist sequences of compact $\RCD(K, N)$ spaces $X_i, Y_i$ with (\ref{eq:diam}), $\tau>0$ and a sequence of maps $\psi_i:(B_r(x_i), x_i) \to (B_r(y_i), y_i)$ such that $\psi_i$ is an $\epsilon_i$-approximation of the eigenfunctions for some $\epsilon_i \to 0^+$ and that $\Phi_i$ is not a $\tau$-approximation of the heat kernel. As in the proof above, after passing to a subsequence, we know;
\begin{itemize}
    \item $X_i, Y_i$ mGH-converge to compact $\RCD(K, N)$ spaces $X, Y$, respectively;
    \item there exists an isometry $\psi:(B_r(x), x) \to (B_r(y), y)$ as metric spaces for some $x \in X$ and some $y \in Y$ such that it is an $\epsilon$-approximation of the eigenfunctions for any $\epsilon>0$ and that it is not a $\tau$-approximation of the heat kernel, 
\end{itemize}
where we immediately used the spectral convergence established in \cite[Th. 7.8]{GMS} (see also \cite{H24, KS}). However, it is a contradiction because of the expansion of the heat kernels of $X, Y$ by eigenfunctions.

Similarly we can obtain (2) via \cite{GMS}, where we omit the proof.
\end{proof}
Let us focus on the class
\begin{equation}
    \mathcal{M}(n, D, v)
\end{equation}
defined as the set of all isometry classes of closed Riemannian manifolds $M$ with $|\mathrm{Ric}(M)| \le 1$, $\mathrm{diam}(M) \le D$ and $\mathrm{Vol}(M) \ge v$.
Note that though the following is stated in terms of the heat kernel, we can replace ``heat kernel'' by ``eigenfunctions'' because of Proposition \ref{prop:equiv}.
\begin{theorem}[Continuity of Gel'fand's inverse problem under bounded Ricci curvature]\label{thm:function}
For all $n \in \mathbb{N}$ and $\epsilon, r, v, D \in (0, \infty)$, 
if for two closed Riemannian manifolds $M_i \in \mathcal{M}(n, D, v)$, $i=1,2$, there exists an $\epsilon$-approximation of the heat kernel $\psi:B_r(x_1) \to B_r(x_2)$ for some $x_i \in M_i$, 
then
\begin{equation}
\dist_{\mathrm{GH}}(M_1, M_2)<\varepsilon_{n, r, v, D}(\epsilon).
\end{equation}
Furthermore, $\psi$ is obtained as the restriction of an $\varepsilon_{n, r, v, D}(\epsilon)$-mGH approximation $\Psi: M_1 \to M_2$ to $B_r(x_1)$ with the almost uniqueness property. 
\end{theorem}
\begin{proof}
Let us prove the first statement only because the remaining ones are also a direct consequence of the proof with Theorem \ref{uniqueness-bijection}. 
If it is not the case, then, together with the continuity of the heat kernels with respect to the pmGH convergence (\ref{eq:heat conv}), we can find
\begin{itemize}
\item sequences of closed Riemannian manifolds $\{M_{j, k}\}_k$ of dimension $n$, $j=1,2$, with uniform two-sided Ricci curvature bounds, have non-collapsed Gromov-Hausdorff limit spaces $(X_j, \dist_{X_j})$ which are not isometric to each other;
\item there exists an isometry $\psi:B_r(x_1) \to B_r(x_2)$ for some $x_i \in X_i$ such that $\psi$ preserves the heat kernels for all small time.
\end{itemize}
Then we can get a contradiction via Theorem \ref{uniqueness-bijection}, together with the $C^{1,\alpha}$-regularity result on the regular sets $\mathcal{R}_{X_j, n}$ obtained in \cite{CC1}.
\end{proof}
Note that the theorem above can be generalized to the compactification of $\mathcal{M}(n, D, v)$, denoted by $\overline{\mathcal{M}(n, D, v)}$, by the proof.
By similar arguments as in the proof above, thanks to Theorem \ref{thm:unique almost} with the stability of a uniform CAT condition with respect to the pGH convergence, we obtain the following.
\begin{theorem}[Continuity of Gel'fand's inverse problem under mixed curvature bounds]\label{thm:mixed}
    For all $K, \kappa \in \mathbb{R}, N \in [1, \infty)$ and $\epsilon, r, D \in (0,\infty)$, if two compact metric-measure spaces $(X_i, \dist_{X_i}, \meas_{X_i})$, $i=1,2$, satisfy that they are $\RCD(K, N)$ and $\mathrm{CAT}(\kappa)$ space, that $\mathrm{diam}(X_i), \meas_{X_i}(X_i) \in [D^{-1}, D]$, and that there exists an $\epsilon$-approximation of the heat kernel $\psi:B_r(x_1) \to B_r(x_2)$ for some $x_i \in X_i$, 
then
\begin{equation}
\dist_{\mathrm{mGH}}(X_1, X_2)<\varepsilon_{K, \kappa, N, r, D}(\epsilon).
\end{equation}
Furthermore, $\psi$ is obtained as the restriction of an $\varepsilon_{n, r, v, D}(\epsilon)$-mGH approximation $\Psi: X_1 \to X_2$ to $B_r(x_1)$ with the almost uniqueness property. 
\end{theorem}
Compared with Theorem \ref{thm:function}, it should be emphasized that collapsing cases are included in Theorem \ref{thm:mixed}, namely the dimensions of $X_i$ may be different as $[0,1]^n\times [0,\epsilon]^k \to [0,1]^n$ as $\epsilon \to 0$. 

Finally let us discuss the case under lower Ricci curvature bound, where we note that the proof of the following does not use the unique continuation for almost smooth spaces (namely it is enough to use it for closed manifolds). It is worth mentioning that this is not true if $M$ is replaced by an element of $\overline{\mathcal{M}}$ because a topological change occurs even for non-collapsed Einstein $4$-manifolds, see for instance \cite{KT, Page}.

\begin{theorem}[Canonical diffeomorphism via Gel'fand's inverse problem]\label{thm:diff1}
Let $M$ be a closed Riemannian manifold of dimension $n$.
For all $K \in \mathbb{R}, D>0$ and $\epsilon>0$, there exists $\delta>0$ such that the following holds. If an $\RCD(K, n)$ space $X$ with $\mathrm{diam}(X) \le D$ has 
a $\delta$-approximation of the heat kernel $\psi:(B_r(x), x) \to (B_r(y), y)$ for some $x \in X$ and some $y \in M^n$, then  $X$ is non-collapsed up to multiplication of a positive constant to the reference measure, and there exists an $\epsilon$-GH approximation $\Psi:X \to M$, which is a homeomorphism, such that 
\begin{equation}
\dist_{M}(\psi(z), \Psi (z))<\epsilon, \quad \text{for all $z \in B_r(x)$},
\end{equation}
and that
\begin{equation}
    (1-\epsilon)\dist_{X}(\tilde x, \tilde y)^{1+\epsilon} \le \dist_{M}(\Psi(\tilde x), \Psi(\tilde y)) \le (1+\epsilon)\dist_{X}(\tilde x, \tilde y)
\end{equation}
holds for all $\tilde x, \tilde y \in X$. Moreover, if $X$ is smooth, then $\Psi$ can be obtained as a diffeomorphism.
 \end{theorem}
\begin{proof}
The proof is essentially same to that of Theorem \ref{thm:function}, noticing a canonical diffeomorphism result is established in \cite{HP}.
\end{proof}
Note that in the last two theorems above, we can replace ``$\delta$-approximation of the heat kernel'' by a weaker one, (\ref{eq:heat comparison}), (with $\epsilon=\delta$), due to Lemma \ref{lem:surjective}, and that in the final theorem, not only the almost uniqueness holds, but also it is homotopically unique because for two such extensions $\Psi_i\,(i=1,2)$ of $\psi$, $\Psi_1 \circ \Psi_2^{-1}$ is $C^0$-close to an isometry of $M$.

\medskip

\appendix
\section{Unique continuation for Lipschitz metric}
\label{sec-Lip}

The Carleman estimates for the Laplace-Beltrami operator on Riemannian manifolds with Lipschitz metric were proved in \cite{LL21}. Unique continuation for non-stationary Schr{\"o}dinger operator with time-independent Lipschitz metric was proved in \cite{FLL}.
In this appendix we show that time-independent Lipschitz metric is also sufficient for the Tataru-type unique continuation for the wave operator, although this may be not surprising to experts of this subject.

\begin{theorem} \label{uc-manifold-Lipschitz}
Let $(M,g)$ be a connected closed Riemannian manifold with Lipschitz metric $g$. Let $T>0$ and $V\subset M$ be an open subset. If $u\in H^1(M\times (-T,T))$ is a weak solution of $(\partial_t^2-\Delta_g)u=0$ and $u|_{V\times (-T,T)}=0$, then $u=0$ in $K(V,T)$.
\end{theorem}

The unique continuation property is a direct consequence of the Carleman-Tataru estimate (Proposition \ref{prop-Carleman-Lip}) for Lipschitz metric.
Once Proposition \ref{prop-Carleman-Lip} is proved, the classical argument gives local unique continuation across non-characteristic hypersurfaces, see e.g. \cite[Sec. 2.5.33]{KKL}. 
Then one can propagate the local unique continuation to global unique continuation in the optimal domain $K(V,T)$ by following \cite[Sec. 3.4.2]{KKL}. Note that the construction of non-characteristic hypersurfaces via distance functions uses smooth metric in \cite{KKL}, and in our case we can approximate the Lipschitz metric with smooth metric (in fixed coordinate charts). 

\begin{proposition} \label{prop-Carleman-Lip}
Let $U\subset M$ be a local coordinate chart at $x_0\in M$ where the metric is Lipschitz, and let $\Box=\partial_t^2-\Delta_g$ be the wave operator on $\mathbb{R}\times U$.
Let $y_0=(t_0,x_0)\in \mathbb{R}\times U$.
Then there exists a quadratic polynomial $\phi$ and constants $C,r,d,\tau_0,\epsilon_0$ such that for $\tau\geq \tau_0$ and $\epsilon<\epsilon_0$,
\begin{equation}\label{eq-Carleman-Lip}
\tau\| e^{-\epsilon D_0^2/2\tau} e^{\tau\phi} u\|^2_{1,\tau} \leq C\|e^{-\epsilon D_0^2/2\tau} e^{\tau\phi} \Box u\|^2_{L^2}+ Ce^{-d \tau} \|e^{\tau \phi}u\|^2_{1,\tau},
\end{equation}
for any $u\in H_0^1(B_r(y_0))$ and $\Box u\in L^2(B_r(y_0))$.
\end{proposition}

\begin{proof}
The proof for Lipschitz metric essentially follows from \cite{FO2} after small modifications, and we only record here the adjustment specific to Lipschitz metric.
We follow the notations in \cite{FO2}.
Denote the semi-Riemannian metric on $\mathbb{R}\times M$ by $\overline{g}=-dt^2+g$ and its Laplacian is $\Delta_{\overline{g}}=-\Box=-\partial_t^2+\Delta_g$. 
Denote the semi-Riemannian quadratic form by $G(X)=(X,X)_{\overline{g}}$, and by $\nabla, D^2$ the gradient and Hessian with respect to $\overline{g}$.
As in \cite[Lem. 2.1]{FO2}, we use the quantities
\begin{eqnarray*}
&a=\sigma-\Delta_{\overline{g}} \ell, \quad q=G(\nabla \ell)-\Delta_{\overline{g}} \ell, \quad L=2\nabla \ell, \\
&Q_{\pm}(X)=2 D^2 \ell(X,X)\pm a G(X), \quad e^{\ell} \Delta_{\overline{g}} (e^{-\ell} v)= (\Delta_{\overline{g}}-L+q)v.
\end{eqnarray*}
Here $\ell=\tau \phi$ where $\phi$ is a quadratic polynomial of (chosen) coordinate functions independent of $\tau$, and $\sigma$ is a Lipschitz function in the scale of $\tau$.

Since we work in the coordinate chart $U\subset M$ where the metric is Lipschitz,
observe that $a\in L^{\infty}(\mathbb{R}\times U)$, so one needs to modify \cite[Lem. 2.1]{FO2} such that no differentiation hits $a$ in the remainder $R$. For $v\in C^{1,1}$, instead of estimating $|e^{\ell} \Delta_{\overline{g}} (e^{-\ell} v)|^2$, we consider $|e^{\ell} \Delta_{\overline{g}} (e^{-\ell} v) - av |^2$. We claim that
\begin{eqnarray}\label{modified-equality}
\frac{1}{2} |e^{\ell} \Delta_{\overline{g}} (e^{-\ell} v) - av |^2 &=& Q_+(\nabla v) +Q_{-}(\nabla \ell) v^2 + \textrm{div} \,B +R \nonumber \\
 &&+\frac12 |(\Delta_{\overline{g}}+G(\nabla \ell))v|^2 + \frac12 |(L+\sigma)v|^2,
\end{eqnarray}
where $B$ and $R$ are both Lipschitz:
\begin{equation}
B=-((L+\sigma)v) \nabla v + (G(\nabla v)-G(\nabla \ell)v^2) \nabla \ell, \quad R=v (\nabla \sigma, \nabla v)_{\overline{g}}.
\end{equation}
This can be seen from eq. (2.4) in \cite{FO2} by splitting as $(\Delta_{\overline{g}}+q+\sigma-a)v$ and $(L+\sigma)v$. The first cross term $-\Delta_{\overline{g}} v (L+\sigma)v$ is the same. The second cross term $(q+\sigma-a)v (L+\sigma)v=G(\nabla \ell)v (L+\sigma)v$ is Lipschitz. A direct computation leads to \eqref{modified-equality} where the right-hand side makes sense in $L^{\infty}$.

\smallskip
Then we consider the conjugated operator $\Box_{\ell}=e^{\ell} \Box e^{-\ell}$, and compute the operator $\Box_{\ell,\epsilon}$ such that $\Box_{\ell,\epsilon} e^{-\epsilon D_0^2/2\tau} = e^{-\epsilon D_0^2/2\tau} \Box_{\ell}$. Following the same computations in \cite[Sec. 2.2.2]{FO2}, 
$$e^{-\frac{\epsilon D_0^2}{2\tau}} \Box_{\ell} v=\Box e^{-\frac{\epsilon D_0^2}{2\tau}} v - 2 e^{-\frac{\epsilon D_0^2}{2\tau}} (\nabla \ell, \nabla v)_{\overline{g}} - e^{-\frac{\epsilon D_0^2}{2\tau}} (\Box \ell)v+ e^{-\frac{\epsilon D_0^2}{2\tau}} v(\nabla \ell,\nabla \ell)_{\overline{g}}.$$
Note that in eq. (2.12) in \cite{FO2}, an explicit weight function $\ell=\tau\phi$ was chosen with respect to the Minkowski metric. For general metric, $\Box \ell$ is linear in time $t$ since the Laplacian involves the first-order derivative of the metric. Hence,
$$e^{-\frac{\epsilon D_0^2}{2\tau}} (\Box \ell)v = (\Box \ell) e^{-\frac{\epsilon D_0^2}{2\tau}} v +\frac{\epsilon}{\tau}(\Box \partial_t \ell) \partial_t e^{-\frac{\epsilon D_0^2}{2\tau}} v.$$
Thus, an additional term $A_5$ is needed in the general case:
\begin{equation}
\Box_{\ell,\epsilon}=\Box_{\ell}-2A_1+A_2+A_3+A_4+A_5,
\end{equation}
where
\begin{eqnarray}
&& A_1 v =\frac{\epsilon}{\tau} (\nabla \ell_t, \nabla v_t)_{\overline{g}}, \quad A_2 v =\frac{2\epsilon}{\tau} v_t (\nabla \ell_t, \nabla \ell)_{\overline{g}}, \quad A_3 v = \frac{\epsilon^2}{\tau^2} G(\nabla \ell_t) v_{tt}, \nonumber \\
&& A_4 v =\frac{\epsilon}{\tau} G(\nabla \ell_t) v, \quad A_5 v =-\frac{\epsilon}{\tau} (\Box \ell_t) v_t.
\end{eqnarray}
Then following the estimates in \cite[Prop. 2.3]{FO2}, we consider
\begin{equation}
(\Box_{\ell,\epsilon}-A_4-A_5-a)v=(\Box+G(\nabla \ell)-2A_1+A_3)v - (L-A_2+\sigma)v,
\end{equation}
and the cross terms of its square are
$$-\Box v (L+\sigma)v -G(\nabla \ell) v (L+\sigma)v + \widetilde{R},$$
where
\begin{equation} \label{eq-R-remainder}
\widetilde{R}=2A_1 v Lv-2A_1v A_2 v+ 2\sigma v A_1 v-A_3 v Lv+A_3 v A_2 v - \sigma v A_3 v+\Box v A_2 v+G(\nabla \ell) v A_2 v.
\end{equation}
Applying \eqref{modified-equality} to the first two cross terms and integrating on $\mathbb{R} \times U\subset \mathbb{R}^{n+1}$, we have
\begin{equation}
\int Q_{-}(\nabla \ell)v^2+ \int Q_{+}(\nabla v)+\int R + \int \widetilde{R} \leq \int \big|(\Box_{\ell,\epsilon}-A_4-A_5-a)v \big|^2.
\end{equation}
Since $\ell,\sigma$ is in the scale of $\tau$, we see that
$$\int |(a+A_4+A_5)v|^2 \leq C\tau^2\int v^2 + C\epsilon^2 \int v_t^2, \quad \big|\int R \big| \leq C\tau^2 \int v^2 + C\int |\nabla v|^2.$$

Next, we estimate the cross terms $\widetilde{R}$ in \eqref{eq-R-remainder}. 
Using integration by parts and $\nabla \ell_{tt}=0$ (since $\ell=\tau\phi$ is a quadratic polynomial of the coordinate functions), we estimate the first term:
\begin{eqnarray*}
\int A_1 v Lv &=&\frac{2\epsilon}{\tau} \int (\nabla \ell_t,\nabla v_t) (\nabla \ell, \nabla v)=-\frac{2\epsilon}{\tau}\int (\nabla \ell_t,\nabla v)^2 - \frac{2\epsilon}{\tau}\int (\nabla \ell_t, \nabla v)(\nabla \ell, \nabla v_t).
\end{eqnarray*}
To proceed, observe that
\begin{eqnarray*}
&&\int (\nabla \ell_t,\nabla v_t) (\nabla \ell, \nabla v) - \int (\nabla \ell,\nabla v_t) (\nabla \ell_t, \nabla v) = \int v_t \, \textrm{div} \big( (\nabla \ell, \nabla v) \nabla \ell_t \big) - \int v_t \, \textrm{div} \big( (\nabla \ell_t, \nabla v) \nabla \ell\big) \nonumber \\
&&= \int v_t\Big((\Delta \ell_t) (\nabla \ell,\nabla v) - (\Delta \ell)(\nabla \ell_t,\nabla v) \Big)+ \int v_t \big( \nabla (\nabla \ell,\nabla v), \nabla \ell_t \big)-\int v_t \big( \nabla (\nabla \ell_t,\nabla v), \nabla \ell \big) \nonumber \\
&&= \int v_t\Big((\Delta \ell_t) (\nabla \ell,\nabla v) - (\Delta \ell)(\nabla \ell_t,\nabla v) \Big) + \int v_t \big([\nabla \ell_t,\nabla \ell],\nabla v \big).
\end{eqnarray*}
Thus, we have
$$\int A_1 v Lv=-\frac{\epsilon}{\tau} \int (\nabla \ell_t,\nabla v)^2+\frac{\epsilon}{\tau} \int v_t\Big((\Delta \ell_t) (\nabla \ell,\nabla v) - (\Delta \ell)(\nabla \ell_t,\nabla v) \Big) + \frac{\epsilon}{\tau}\int v_t \big([\nabla \ell_t,\nabla \ell],\nabla v \big).$$
The other terms in $\widetilde{R}$ are simpler.
\begin{eqnarray*}
\int A_1 v A_2 v = \frac{2\epsilon^2}{\tau^2} \int (\nabla \ell_t, \nabla v_t)(\nabla \ell_t, \nabla \ell)v_t = -\frac{\epsilon^2}{\tau^2} \int \textrm{div} \big( (\nabla \ell_t,\nabla \ell) \nabla \ell_t \big) v_t^2.
\end{eqnarray*}
\begin{eqnarray*}
\int A_3 v Lv &=& \frac{2\epsilon^2}{\tau^2} \int G(\nabla \ell_t) v_{tt} (\nabla \ell,\nabla v) =-\frac{2\epsilon^2}{\tau^2} \int G(\nabla \ell_t) v_t \Big((\nabla \ell_t, \nabla v) + (\nabla \ell, \nabla v_t)\Big) \\
&=& -\frac{2\epsilon^2}{\tau^2} \int G(\nabla \ell_t) v_t (\nabla \ell_t, \nabla v) + \frac{\epsilon^2}{\tau^2} \int \textrm{div} \big( G(\nabla \ell_t) \nabla \ell\big)v_t^2.
\end{eqnarray*}
\begin{eqnarray*}
\int A_3 v A_2 v= \frac{2\epsilon^3}{\tau^3} \int G(\nabla \ell_t) (\nabla \ell_t,\nabla \ell) v_t v_{tt}=-\frac{\epsilon^3}{\tau^3} \int G(\nabla \ell_t)^2 v_t^2.
\end{eqnarray*}
\begin{eqnarray*}
\int \Box v A_2 v &=& -\frac{2\epsilon}{\tau}\int \Delta_{\overline{g}} v (\nabla \ell_t,\nabla \ell) v_t = \frac{2\epsilon}{\tau} \int \Big( \nabla v, \nabla \big( (\nabla \ell_t,\nabla \ell) v_t \big)\Big) \\
&=& \frac{2\epsilon}{\tau} \int \big( \nabla v, \nabla (\nabla \ell_t,\nabla \ell) \big) v_t - \frac{\epsilon}{\tau} \int G(\nabla v) G(\nabla \ell_t).
\end{eqnarray*}
For the terms involving $\sigma$,
\begin{eqnarray*}
\int \sigma v A_1 v=\frac{\epsilon}{\tau} \int \sigma v (\nabla \ell_t,\nabla v_t) =-\frac{\epsilon}{\tau} \int \textrm{div}\big(\sigma v \nabla\ell_t \big) v_t.
\end{eqnarray*}
\begin{eqnarray*}
\int \sigma v A_3 v=\frac{\epsilon^2}{\tau^2} \int \sigma v G(\nabla\ell_t) v_{tt} =-\frac{\epsilon^2}{\tau^2} \int G(\nabla\ell_t)(\sigma_t v v_t+\sigma v_t^2).
\end{eqnarray*}
Since $\tau,\sigma$ are in the scale of $\tau$ and the metric is Lipschitz, all the quantities above are bounded by
$C\epsilon\tau \int (v^2+|\nabla v|^2).$
For the last term,
\begin{eqnarray*}
\int G(\nabla \ell)v A_2 v= \frac{2\epsilon}{\tau} \int G(\nabla \ell) v (\nabla \ell_t,\nabla \ell)v_t = -\frac{\epsilon}{\tau} \int \Big(2(\nabla \ell_t,\nabla \ell)^2+G(\nabla\ell)G(\nabla\ell_t)\Big)v^2,
\end{eqnarray*}
which is bounded by $C\epsilon\tau^3 \int v^2$. 
Thus, we have
\begin{equation}
\big|\int \widetilde{R} \big| \leq C\epsilon \big(\tau^3 \int v^2+\tau \int |\nabla v|^2 \big).
\end{equation}

Finally, the weight function $\phi$ is chosen as a quadratic polynomial (in a chosen coordinate chart) by the usual convexification procedure and taking the Taylor expansion up to the second order, see e.g. \cite[Sec. 2.5.12]{KKL} and \cite[Sec. A.3]{LL21}. More precisely, one can choose a function $\Phi$ such that $(\nabla \Phi,\nabla \Phi)_{\overline{g}}=|\nabla_x \Phi|^2-|\nabla_t \Phi|^2\geq c>0$, and take $\tilde{\phi}=e^{\lambda\Phi}-1$.
If we take $\ell=\tau \tilde{\phi}$, both $Q_{+}(X)$ and $Q_{-}(\nabla \ell)$ are dominated by the highest order $\lambda$-terms for sufficiently large $\lambda$. 
Hence, by \cite[Cor. A.10]{LL21}, for sufficiently large $\lambda$,
\begin{equation} \label{weight-convexity}
\tau^{-1} Q_+(X)\geq c\lambda^2 e^{\lambda \Phi}(X,X)_{\overline{g}}, \quad \tau^{-3} Q_{-}(\nabla \ell)\geq \frac12\lambda^4 e^{3\lambda \Phi}c^2>0.
\end{equation}
Thus one can take the actual quadratic weight function $\phi$ to be the Taylor expansion of $\tilde{\phi}$ up to the second order, so that the estimates \eqref{weight-convexity} holds (with a different constant).
Combining all the estimates above, choosing $\tau$ sufficiently large and $\epsilon$ sufficiently small, the same estimate as \cite[Prop. 2.3]{FO2} follows and consequently the Carleman-Tataru estimate in \cite[Thm. 2.5]{FO2} holds for compactly supported smooth functions $u$.
The estimate extends to compactly supported $H^1$-functions $u$ satisfying $\Box u\in L^2$ by approximation arguments, see \cite[Sec. 2.5.32]{KKL}.
\end{proof}

\medskip

\end{document}